\documentclass[reqno]{amsart}
\usepackage{graphicx} 
\usepackage{amsthm}
\usepackage{amsmath,mathtools}
\usepackage{amssymb}
\usepackage{epsfig}
\usepackage{mathrsfs}
\usepackage[labelformat=simple]{subcaption}

\usepackage{pdfsync}
\usepackage{hyperref}
\usepackage{enumitem}
\usepackage{fullpage}
\usepackage[normalem]{ulem}
\usepackage{comment}
\usepackage[dvipsnames]{xcolor}

\def\Tcal{{\mathcal T}}

\def\Ncal{{\cal N}}

\def\Pcal{{\mathcal{P}}}
\def\Rcal{{\cal R}}
\def\Scal{{\mathcal{S}}}
\def\Ical{{\mathscr I}} 

\def\Ncal{{\mathcal{N}}}
\def\Rcal{{\mathcal{R}}}
\def\UN{{\Ncal}}
\def\UNbar{{\overline{\UN}}}

\definecolor{mygreen}{rgb}{0.192,0.64,0.05}

\def\R{\mbox{$\mathbb{R}$}}

\def\N{{\mathbb{N}}}

\def\1{\mbox{\boldmath$1$}}

\def\uo{{\overline{u}}}

\def\implies{{\Rightarrow}}

\def\e{{\mathop {\rm e}\nolimits}}

\def\eps{\varepsilon}

\def\dis{\displaystyle}
\newcommand{\inter}[1]{\mathrm{int}(#1)}

\def\e{\mathop{\rm e}\nolimits}

\def\mint{{-}\hspace{-2.5ex}\int}

\def\loc{{\mathop {\rm loc}\nolimits}}

\def\ub{{\overline u}}

\DeclarePairedDelimiterX{\inp}[2]{\langle}{\rangle}{#1, #2}

\DeclareMathOperator{\ext}{ext}
\DeclareMathOperator{\co}{co}
\DeclareMathOperator{\cone}{cone}
\DeclareMathOperator*{\argmax}{argmax}

\newtheorem{definizione}{Definizione}[section]

\newtheorem{lemma}[definizione]{Lemma}

\newtheorem{theorem}[definizione]{Theorem}

\newtheorem{definition}[definizione]{Definition}
\newtheorem{proposition}[definizione]{Proposition}
\newtheorem{corollary}[definizione]{Corollary}

\newtheorem*{notation}{Notation}

\newtheoremstyle{myremark}
  {}   
  {}   
  {\normalfont} 
  {}   
  {\bfseries}   
  {.}
  { }
  {}

\theoremstyle{myremark}
\newtheorem{remark}[definizione]{Remark}
\newtheorem{example}[definizione]{Example}

\date{}

\begin{document}

\title{
On Zermelo's  planar navigation problem  for  convex bodies,\\
and implications for non-convex optimal  routing}

\author{}

\author{Matteo Della Rossa}
\author{Lorenzo Freddi} 
\author{Mattia Pinatto}
\address[Matteo Della Rossa]{Dipartimento di Elettronica e Telecomunicazioni, Politecnico di Torino, corso Duca degli Abruzzi 24, 10129 Torino, Italy }
\email{matteo.dellarossa@polito.it}
\address[Lorenzo Freddi, Mattia Pinatto]{Dipartimento di Scienze Matematiche, Informatiche e Fisiche,  Universit\`a di Udine, via delle Scienze 206, 33100 Udine, Italy}
\email{lorenzo.freddi@uniud.it, pinatto.mattia@spes.uniud.it}

\begin{abstract} 
We study a generalized version of Zermelo’s navigation problem where the set of admissible velocities is a general  compact convex  set, replacing the classical Euclidean ball. After establishing existence results under the natural assumption of weak currents, we derive necessary optimality conditions via Pontryagin’s maximum principle and convex analysis. Consequently, in the planar case, the domain of any optimal control is shown to be partitioned into regular and singular regimes. In the former, 
the optimal control is regular and,  under mild additional assumptions, satisfies a Zermelo-like navigation equation while in the latter it is largely undetermined. A necessary condition that can exclude singular regimes is stated and proved, providing a useful tool in applications.  In regular regimes our results extend the   classical Zermelo navigation equation to general convex control sets within a non-parametric setting. Furthermore, we discuss direct applications to the case of a non-convex control set. 
As an application, we develop the relevant case of an affine current. The results are illustrated with examples relevant to sailing and ship routing with asymmetric or sail-assisted propulsion, including the  presence of waves.
\end{abstract}

\maketitle


\medskip


\section{Introduction}

  In 1931, Ernst Zermelo \cite{ZermeloII} posed his   classical navigation   problem: ``in an unlimited plane, where the distribution of the wind is given by means of a vector field depending on position and time, an airplane moves with constant velocity relative to the air. How must the airplane be directed in order to reach a point 
$B$, in the shortest possible time, starting from a fixed point 
$A$?'' This problem and its generalizations to three or more dimensions has been studied by Zermelo himself and other renowned mathematicians including Levi-Civita \cite{LC1931}, von Mises \cite{vM1931},  Carath\'eodory \cite{Carath1967} and Mani\`a \cite{Mania1937}. 
For further details on the contributions of these authors, we refer to Mani\`a’s paper \cite{Mania1937}.   
However, all of them adhere to  Zermelo's original setting regarding the choice of a sphere as control set $U$ of admissible velocities.  
A relaxed, but in fact equivalent, formulation of the problem is obtained by taking $U$ to be a ball (the convex hull of the sphere), see \cite[Theorem 2.2]{CFM06}. 
In this paper we extend the analysis by allowing the control set $U$ to be a general convex set and, in the case of planar navigation, we recover Zermelo's navigation equation (ZNE, for short) in the particular case in which $U$ is a ball.

More recently, several variants of Zermelo's problem have been considered and studied in the case of  particular flows (see, e.g.~\cite{BBBCG2019}) or navigation on manifolds (see, e.g.~\cite{Serres2009,K2019,BCW2023}). 
See also \cite{PBG2022} for applications to autonomous navigation of microswimmers.

Clearly, the choice of the airplane as a navigation vehicle is merely illustrative. In fact,
in the case 
$n=2$ the problem arises even more naturally in  planar navigation of a ship in the presence of current. The case where 
$U$ is a ball corresponds to motor navigation. The ball is deformed into a shifted ellipse as an effect of the action of waves (see, e.g.,  \cite[Fig. 8]{sake}). More complex velocity sets (including those considered in this work) characterize the performance of sail-assisted cargo ships, where the use of (usually rigid)  sails or wings primarily serves the purpose of energy saving and $CO_2$ 
 emission reduction (see, e.g., \cite{RVHB2021}). 
As regards the planar case, it is worth recalling that there is a long-standing tradition of approaches to the optimal control of low-dimensional systems, starting from the works of Sussmann, Krener, Schättler, and others \cite{Sussmann1987_1,KS1989}, and based on a systematic use of Lie bracket conditions. 
More recent contributions, such as \cite{BCW2023}, also fall within this line of research and provide a thorough analysis. 
These approaches typically investigate accessibility and optimality through the Lie algebra rank condition and the geometry of the exponential map, and are naturally tailored to single-input control systems. 
When adapted to the classical Zermelo framework, to the best of our knowledge, they allow at most for smoothly parameterized control sets, such as ellipses (Example \ref{ex_shift_ellip} falls in this setting, and the subsequent Remark \ref{rem:BCW23comparing} links to the approach of \cite{BCW2023}). 
By contrast, in our setting the control variable ranges over a full compact convex set $U \subset \mathbb{R}^2$. 
Whenever $0 \in \operatorname{int}(U)$, the system is locally of full rank in a trivial way, so that accessibility issues addressed via Lie bracket conditions do not play a central role. 
The phenomena we investigate, namely the partition into regular and singular regimes and the possible emergence of tack points, are therefore not related to a loss of rank of the exponential map or to conjugate points in the sense of geometric control theory. 
Rather, they arise from the interaction between the drift and the convex geometry of $U$, especially in the presence of non-strict convexity of its boundary. 
This justifies the use of convex-analytic tools as a natural framework in our study,  that  focuses  on allowing general control sets, as compact (non-)convex sets. A contribution in this direction has been given in \cite{CFS2000}, where an additional necessary optimality condition requires that the  costates be included in the superdifferential of the minimum time function.  Other relevant  contributions come from a recent  different approach that uses Finsler metrics and differential geometry  
(see \cite{FB2022,FB2024,MPRR2025} and references therein). In contrast, our method relies on optimal control tools, specifically Pontryagin Maximum Principle and convex analysis, similar to \cite{CFS2000}. However, unlike previous approaches, we do not require prior knowledge of the minimum-time function, which depends heavily on the configuration of the initial set $A$ and target region $B$, as well as on the control set.   In fact, our  approach naturally allows for a general choice of $A$ and $B$, that are assumed to be single points only when necessary or not restrictive.

Zermelo's navigation problem can be formulated as a {\em minimum-time control problem}  
from the region $A$ to the region $B$ with dynamics 
\begin{equation}\label{eq_state_equation}
\dot x(t)=u(t)+
        s(t,x(t)),
\end{equation}
where $u(t)\in U$ is the control and the drift term $
        s(t,x)$ describes the {\em current}. 
The problem 
consists in determining the control function $u$ 
that minimizes the travel time $t_f$  
over the space of all admissible trajectories $x$ leading from $A$ to $B$.  

After proving the existence of optimal solutions, we derive necessary optimality conditions via the Pontryagin Maximum Principle. Our approach is particularly effective in dimension two, although some of the results hold in arbitrary finite dimension.  Indeed, in the planar case, classical tools  of convex analysis are used to prove that, 
when the conjugate variable is not 
normal to a segment on the boundary (regular regime),
the optimal control is regular  (Lemma~\ref{lm:geo_conv}). This allows us  to 
eliminate the co-states  and obtain a differential equation, involving only $u$ and $x$, that generalizes the classical Zermelo navigation equation; this is stated in Theorem \ref{th:char_u_convex-case}. 
Such navigation equation can be coupled with the state equation to characterize the time optimal controls.  A remarkable fact about our derivation is that, even if obtained in a non-autonomous setting, it does not rely on conservation nor on the evolution law of the Hamiltonian. This observation shows that the Zermelo navigation equation is not a consequence of energy conservation, but rather stems from the intrinsic geometric structure of the problem. Moreover   it  arises in a non-parametric setting, that is, without {\em a priori} imposing any specific parametrization of the control. This flexibility is particularly relevant in applications, where the choice of the parametrization naturally depends on the geometry of $U$. 
For instance, in the case in which $U$ is a ball and polar coordinates are used, we recover the classical, well known, Zermelo navigation equation. Nevertheless, our setting extends its validity to strictly convex control sets. 
Conversely, when the control set is non-strictly convex, the Zermelo equation holds only within specific time intervals, which we refer to as {\em regular regimes}. 
When, on the contrary, the conjugate variable stays along the normal to a segment on the boundary then Zermelo equation does not hold and we say that a {\em singular regime} occurs.  Generally  speaking,   regular  regimes may  alternate with  singular ones or be separated by {\em tack points}  where the trajectory's derivative may jump.  This phenomenon can arise only when the control set is not strictly convex and, in the general convex case, we are able to provide a necessary condition (on the gradient of the current along the trajectory)  under which   a singular regime may occur (Theorem~\ref{th:sing_reg_NC}). 
This tool 
is particularly effective when the current is position-affine since, in such case,  $\nabla_x s$ is independent of $x$. A detailed study is devoted to this type of currents because they can be regarded as first-order approximations of more complex current fields. In particular, in Theorem \ref{th:affine_regimes}, we show that under the assumption of an affine current, regular and singular regimes cannot coexist, and the only possible scenarios are either a single global singular regime or one or more contiguous regular regimes. The latter reduce to only two in the case where the current $x$-gradient has real eigenvalues (which, actually, rules out the presence of a vortex). 

We conclude the analysis by presenting some direct consequences for the case of non-convex control sets and by studying some relevant examples in practical navigation, such as situations involving waves and/or sailing navigation.


\

{\bf Notation.} 
Let $x, y \in \mathbb{R}^n$. We denote by $(x,y):=\{\lambda x+(1-\lambda )y\ : \ \lambda\in(0,1)\}$ 
the (relatively) {\em open segment with endpoints $x$ and $y$} and, whenever $x\ne y$,  $[x, y]:=\overline{(x,y)}$, is the {\em closed} segment. The usual Euclidean scalar product will be denoted by $\inp{x}{y}$, and $|x|$ denotes the corresponding norm. 
 The notation used for spaces of functions is standard. Namely, for maps defined on an interval $I$ and values in a subset $E$ of $\R^d$ we denote by $C^k(I,E)$  the space of functions with continuous $k$-th derivatives,   $L^p(I,E)$ the Lebesgue space of (equivalence classes of) $p$-summable (if $p\in[1,+\infty)$) or essentially bounded (if $p=\infty$) functions,  
and $W^{1,p}(I,E)$ the Sobolev space of
(equivalence classes of) functions that are in $L^p$  together with their distributional derivatives.
The reference to the set $E$ is often omitted when $E=\R$. Moreover, $\R_+:=(0,\infty)$. As is customary, we denote derivatives with respect to time by dots and derivatives with respect to other variables by primes.

\section{Formulation of the problem 
 }\label{sec_Problem_setting}

Given subsets $A,B,U\subset \R^n$  and $s:\R_+\!\times \R^n\to\R^n$ regular enough, we denote by  $\mathcal{P}_{A,B,U,s}$ the \emph{minimum-time control problem} 
already explained in the introduction (see \eqref{eq_state_equation}) 
and that can be also summarized in the following  formulation as a Lagrange or Mayer  optimal control problem
\begin{subequations}
\label{eq:ProblemZermelo}
    \begin{eqnarray}
    \label{eq:cost_funct}
        &\displaystyle\min_{\displaystyle u\in L^0(\R_+,\R^n)}& \int_0^{t_f} \;dt\\
    \label{Eq: Dinamic Equation}
    &&\hspace{-12ex}
    \begin{cases}
        \dot x(t)=u(t)+
        s(t,x(t)),\\
        x(0)=x_0,
    \end{cases}\\
    \label{eq:init_fin_cond}
        &&\hspace{-12ex}x_0\in A,\ x(t_f)\in B,\\
    \label{eq:Constraint}
        &&\hspace{-12ex}u(t)\in U \;\;\text{for a.a.\ } t\in (0,t_f),
    \end{eqnarray}
\end{subequations}
where $L^0(\mathbb{R}_+, \mathbb{R}^n)$ denotes the space of (equivalence classes of) $\R^n$-valued Lebesgue measurable functions defined on the interval $(0,\infty)$. The same problem will be denoted simply by $\Pcal_U$ when the sets $A$, $B$ and the current field $s$ are clear from the context.   
An {\em admissible pair} for  problem $\mathcal{P}_{A,B,U,s}$  is a pair  \((u, x) \in L^0(\mathbb{R}_+, \mathbb{R}^n) \times W^{1,1}_\loc(\mathbb{R}_+, \mathbb{R}^n)\)  that  satisfies conditions \eqref{Eq: Dinamic Equation}--\eqref{eq:Constraint}. 
An admissible pair that achieves the minimum in~\eqref{eq:cost_funct}, if it exists, is called a  {\em solution} to the minimum-time control problem and $u$ is said to be a {\em time-optimal control}.
Our aim is to establish sufficient conditions for the existence of  solutions and characterize the time-optimal controls.  

Throughout the whole paper, the set {\em  $U$ is nonempty and compact} while the ending regions {\em $A$ and $B$ are non-empty, closed and disjoint} sets admitting representations of the form $A=\{x\in\R^n\ :\ \phi_i(x)\le 0,\,i=1,...,k\}$
and 
$B=\{x\in\R^n\ :\ \psi_j(x)\le0,\,j=1,...,m\}$ with $\phi_i,\psi_j\in C^1(\R^n)$ and $k,m\in\N$.

\section{Existence of a solution}\label{sec_existence}

In this section 
we provide conditions 
that, besides the standing assumptions made in Section \ref{sec_Problem_setting} on the sets $U$, $A$, $B$,  
ensure the existence of  a solution to problem $\mathcal{P}_{A,B,U,s}$.

The following hypotheses, that classically ensure global existence and uniqueness of the solution to the Cauchy problem \eqref{Eq: Dinamic Equation} for every control function, are assumed to be satisfied:
\begin{enumerate}
    \item[(s1)] the function $s(t,x)$ is  locally Lipschitz continuous in $x$  uniformly w.r.t.\ $t\in\R_+$;
    \item[(s2)]   there exists a constant $M$ such that {$|s(t,x)|\le M(1+|x|)$ $\forall\,(t,x)\in \R_+\times\R^n$.} 
\end{enumerate}
Given a control function $u\in L^0(\R_+,\R^n)$ and a starting point $x_0\in A$, we denote by $x^{x_0,u}\in W^{1,1}_\loc(\R_+,\R^n)$  the corresponding solution of \eqref{Eq: Dinamic Equation}. The {\em time to reach $B$} starting from $x_0\in A$ with control $u$ is defined by
$\tau(x_0,u):=\inf\{t\in[0,\infty)\ :\ x^{x_0,u}(t)\in B\}$, 
with the convention $\inf\varnothing=+\infty$ (which means that 
$\tau(x_0,u)=+\infty$ if $u$ does not steer $x$ to $B$ in finite time). The functional  $\Tcal:A\to[0,+\infty]$, defined as
\begin{equation}
    \label{def:MTF}
\Tcal(x_0):=\inf\{\tau(x_0,u)\ :\ u\in  L^0(\R_+,U)\},
\end{equation}
is called {\em minimum-time function}. 
When $U$ is convex, $\Tcal$ is lower semicontinuous. This important fact  has been proved in \cite[Proposition 2.1]{CQSP97} for autonomous systems, but the same proof works also in the case of a current depending on $t$.


\begin{lemma}
\label{lem:Tfinite} 
    Suppose that 
    $U$ is a nonempty compact convex subset of $\R^n$  and assumptions {\em(s1)} and {\em (s2)} are satisfied. 
    For every $x_0\in A$ such that $\Tcal(x_0)<+\infty$ the infimum in the definition \eqref{def:MTF} of  $\Tcal(x_0)$ is  attained.
\end{lemma}

\begin{proof}
    The proof is a straightforward application of  \cite[Theorem 5.1.1 and Remark 5.2(i)]{bressan_introduction_2007} because the finiteness of $\Tcal(x_0)$ implies that there is at least one admissible trajectory that joins $A$ to $B$ in finite time.  
\end{proof}


\noindent We are now in position to state the existence theorem in full generality.

\begin{theorem}
\label{th_CS_existence}  
Suppose that 
    $U$ is a nonempty compact convex subset of $\R^n$,   and assumptions {\em(s1)} and {\em (s2)} are satisfied. 
    Assume, moreover, that $A$ is compact and that there exists at least one admissible control driving the system from $A$ to $B$ in finite time. Then, problem $\Pcal_{A,B,U,s}$ 
    admits at least one time-optimal control.
\end{theorem}

\begin{proof}
    Since the minimum-time function $\Tcal$ is lower semicontinuous and  $A$ is compact, by Weierstrass theorem there exists the minimum of $\Tcal$ on $A$. 
    The possibility that the minimum value be $+\infty$, that is $\Tcal\equiv+\infty$, is excluded by the assumption that there exists at least one admissible control steering the solution  from $A$ to $B$ in finite time. 
    Hence,  there exists $x_0\in A$ such that
    $\Tcal(x_0)=\min_A \Tcal<+\infty$ 
    and, by Lemma \ref{lem:Tfinite},  there exists $u\in L^0(\R_+,U)$ such that
    $\Tcal(x_0)=\tau(x_0,u)$. 
    Such control $u$ is a solution of the minimum-time control problem.     
\end{proof}

The last assumption of Theorem \ref{th_CS_existence}, about the existence of  at least one admissible control steering the system from $A$ to  $B$ in finite time is satisfied if the current is not too strong, as the following corollary states.

\begin{corollary}[case of a weak current]\label{cor_weak_stream_ex} Besides the assumptions of Theorem \ref{th_CS_existence} on the sets $U$, $A$, and the current $s$, 
    suppose further that there exists a  Lipschitz connected set  $X$ (see \cite{B2019}) 
    containing $A$ and $B$ in which the following {\em weak current assumption} holds: 
        \begin{enumerate}
                \item[]\hspace{-4ex}{(WC)} there exist $\varepsilon>\delta>0$ such that $B(0,\varepsilon)\subseteq \,U$  and $|s(t,x)|<\delta$  for all $x\in X$ and $t\ge0$.
        \end{enumerate} 
    Then, problem $\Pcal_{A,B,U,s}$ 
    admits a solution.
\end{corollary}


\begin{proof} 
    It follows by the general Theorem \ref{th_CS_existence}  by observing that, under our assumptions, for every $x_0\in A$ there exists an admissible control joining $x_0$ and $B$ in finite time. 
    Indeed, setting $r:=\varepsilon-\delta>0$,  we have {$\overline{B(0,r)}\subseteq U+s(t,x)$ for every $x\in \R^n$ and $t\ge0$.} Since  $X$  is Lipschitz connected, there exist $T_r>0$ and a Lipschitz continuous curve $x:[0,T_r]\to X$ joining $A$ and $B$ such that $\|\dot x\|_\infty\le r$. This implies $\dot x(t)\in U+s(t,x(t))$ for a.e.\ $t\in[0,T_r]$. Thus, the control  $u(t):=\dot x(t)-s(t,x(t))$  drives the state from $A$ to $B$ in finite time. 
\end{proof}

\begin{remark}{The assumption of Lipschitz connection  is always satisfied if $X$ is open and connected. 
    The assumption of dealing with a weak current is not restrictive in the case of maritime transportation, where the vessel's power is sufficient to significantly exceed the speed of currents expected during the navigation.    }
\end{remark}

\section{Pontryagin necessary optimality conditions}\label{sec_Pontryagin}

In this section we provide necessary optimality conditions derived from Pontryagin Maximum Principle (PMP).   
Let us introduce the \emph{Hamiltonian function} $H:\R_+\!\times (\R^{n})^3\times \R\to \R$ defined by $H(t,x,u,p,p_0):=p_0+\inp{p}{u}+\inp{p}{
        s(t,x)}$.

\begin{lemma}[PMP]
\label{lem_PMP} 
Let $(u,x)\in  L^0(\R_+,\R^n)\times W^{1,1}_\loc(\R_+,\R^n)$ be an optimal pair of problem $\Pcal_{A,B,U,s}$
with $s\in C^1(\R_+\times\R^n,\R^n)$. Then, there exist $p_0\in \{0,-1\}$ and $p\in W^{1,1}((0,t_f),\R^n)$ such that  

    \begin{enumerate}[label=(\arabic*)]
        \item \emph{(nondegeneracy)} $p(t)\ne0$ for every $t\in[0,t_f]$;
        \item \emph{(adjoint equation and transversality conditions)} the adjoint function $p$ satisfies 
            \begin{equation} 
            \label{eq:AdjointEquation}
            \begin{cases}
               {\dot p(t)=-(\nabla_x s)^\top\! (t,x(t))\, p(t)}\quad \mbox{for a.e.\ }t\in(0,t_f),\\[1ex]
                p(0)\in  N_A(x(0)),\quad p(t_f)\in -N_B(x(t_f)),
            \end{cases}
        \end{equation}
        where $N_A$ and $N_B$ denote the normal cones to $A$ and $B$ respectively (see \cite[Section 1.4]{Clarke2013}), while $\nabla_x s=({\partial s_i}/{\partial x_j})_{i,j\in\{1,\dots,n\}}$;
        \label{Item:AdjointEquation}
        \item\label{eq_wc} \emph{(Weierstrass condition)} the optimal control $u$ satisfies 
        \[
            u(t)\in \argmax_{u\in U}\,\inp{p(t)}{u}\quad \mbox{ for a.a.\  $t\in (0,t_f)$};
        \]
        \item \label{Pt:evolution_H} {\emph{(evolution law of the Hamiltonian)}  the function $t\mapsto \inp{p(t)}{u(t)+s(t,x(t))}$ coincides a.e.\ with an absolutely continuous function (not renamed) satisfying 
        \begin{equation}\label{eq_cht}
            \frac{d}{dt}\inp{p(t)}{u(t)+s(t,x(t))}=
            \inp{p(t)}{\frac{\partial s}{\partial t}(t,x(t))}
            \quad \mbox{ for a.a.\  $t\in (0,t_f)$};
        \end{equation}}
        {\color{black} if $s$ is independent of time then  \begin{equation}\label{eq_ch}
            p_0+\inp{p(t)}{u(t)+s(x(t))}=0\quad \mbox{ for a.a.\  $t\in (0,t_f)$}.
         \end{equation}
    \label{Item:Constancy}}
    \end{enumerate}
\end{lemma}


\begin{proof}
{The result follows from the standard PMP (see, e.g., \cite[Theorem 2.2]{DO2015})  applied to the augmented system
\[
\dot\tau(t)=1, \qquad \dot x(t)=u(t)+s(\tau(t),x(t)),
\]
with state $(\tau,x)\in\mathbb{R}\times\mathbb{R}^n$ and target $\mathbb{R}\times B$.

Let $q(\cdot)$ be the adjoint variable associated with $\tau$. 
Then the corresponding Hamiltonian relation reads
\[
q(t)+p_0+\langle p(t),u(t)+s(t,x(t))\rangle=0,
\]
with $q(t_f)=0$. Differentiating this identity and using the adjoint equations yields the evolution law of the Hamiltonian.

Moreover, if $p(\bar t)=0$ for some $\bar t$, then $p\equiv 0$ by uniqueness of the adjoint equation. 
Hence $\dot q=0$, so that $q\equiv 0$ since $q(t_f)=0$, and the Hamiltonian relation gives $p_0=0$, contradicting the nondegeneracy condition for the augmented system. 
Therefore $p(t)\neq 0$ for every $t\in[0,t_f]$.}
{\color{black} 

In the autonomous case of a current independent of time, 
the evolution law \eqref{eq_cht} of the Hamiltonian reduces to the usual constancy of the Hamiltonian along the optimal solution. It is well-known that, in fact, for the minimum-time problem such constant is $0$ (see, e.g.,  \cite[Theorem 2.2]{DO2015}).}
\end{proof}



\begin{remark}
\label{Rmk: p nonzero}
{(a) In the case in which $U\subset \R^n$ is convex, by introducing the \emph{support function}  of the set $U$, that is $\sigma_U \colon \R^n \to (-\infty,+\infty]$ defined by $\sigma_U(p) := \sup\{ \langle p, x \rangle\;\colon\;x\in U\}$, the Weierstrass condition  \textit{\ref{eq_wc}} can be rewritten in useful equivalent forms by using the following characterization of the subdifferential of $\sigma_U$ which is classical in Convex Analysis 
(see, e.g., \cite[Proposition 9.1.2, Proposition 9.3.1 and Proposition 9.5.4]{attouch_variational_2006}, \cite[Corollary 8.25]{RockWets09}  and   \cite[Theorem 1.7.4]{Schneider1993}):
\begin{equation}\label{eq:subchar}
            \partial\sigma_U(p)=\argmax\{\langle p,u\rangle\colon u\in U\} =\{u\in U\colon p\in N_U(u)\}
        \end{equation}
for every $p\in\R^n$, where  \begin{equation}\label{def_normal_cone}
N_U(u):=\{p\in\R^n\ :\ \inp{p}{v-u}\le 0\ \forall\,v\in U\}
\end{equation}
denotes the normal cone to $U$ in $u$. Accordingly,     
    Weierstrass condition \textit{\ref{eq_wc}} 
    writes 
    $$
    u(t)\in  \partial\sigma_U(p(t)) \quad \mbox{ for a.a.\  $t\in (0,t_f)$}
    $$
    or,  equivalently,  
        \begin{equation}\label{eq:Weierstrass_NU}
            p(t)\in N_U(u(t))\quad 
            \mbox{ for a.a.\  $t\in (0,t_f)$}.
        \end{equation}

        \vspace{1ex}

        \noindent (b) Since $s\in C^1$ and 
        $x$ is continuous, by the adjoint equation we have $p\in C^1((0,t_f);\R^n)$.
}
\end{remark}

In what follows, we rely on the application of the PMP and therefore assume, unless explicitly stated otherwise, that $s\in C^1(\R_+\!\times\R^n)$.

\section{The case of a convex set \texorpdfstring{$U$}{U}}

\label{Sec: Convex Control Set}

In this section we deal with the case in which  the constant control set $U$, besides being  nonempty and compact, is  convex. By exploring consequences of the Pontryagin necessary conditions  provided by Lemma~\ref{lem_PMP},  in dimension $2$, we characterize the optimal controls as composed of two kind of regimes that will be called, respectively, regular and singular, depending on the evolution of the adjoint variable $p(t)$.    

To this aim, it is useful to recall/introduce some standard definitions. 
A subset $U$ of $\R^n$ is {\em convex} if, for all $x$ and $y$ in $U$, the line segment connecting $x$ and $y$ is included in $U$. According to Clarke's terminology we call {\em convex body} every convex set with nonempty interior. 
An element  $u\in U$ is said to be an {\em extreme point} of $U$  if it does not lie in any open line segment joining two points of $U$. A subset $E$ of $U$ is said to be {\em extreme} if every element of $E$ in an extreme point of $U$.  We denote by $\ext(U)$ the set of all extreme points of $U$. 
The set $U$ is {\em strictly convex} if every point on the line segment connecting $x$ and $y$ other than the endpoints is inside the topological interior of $U$ (denoted, in the sequel, by ${\rm int}(U)$).  
As a consequence, we have that a convex set that is not strictly convex has a boundary that contains at least one open segment. It is well known (Ewald-Larman-Rogers Theorem, \cite[Theorem 2.3.1]{Schneider1993}) that, in dimension $n=2$,  the boundary of a convex set can contain an at most  countable family of disjoint maximal open segments. 
Having this in mind, let us restrict to the dimension $2$ and introduce the following definition which  is made ``ad hoc" for the kind of control sets that we are going to consider.

\begin{definition} 
 Let $U\subset\R^2$ be compact and convex, and let $\Ical$ be at most a countable set of indices. We say that  ${\mathcal{S}} :=\{S_i\}_{i\in \Ical}$ is a  family of {\em disjoint maximal (open) segments} $S_i=(x_i,y_i)$ contained in $\partial U$
 if $x_i,y_i\in\partial U$, $\varnothing\ne S_i\subseteq \partial U$, $S_i\cap S_j=\varnothing$ when $i\ne j$ and $\partial U\setminus\cup_{i\in \Ical}S_i$ is an extreme subset of $U$. If $\#(\Ical)=m\in\N\setminus\{0\}$, we say that {\em $\partial U$ contains  exactly $m$ disjoint maximal open segments}.
 \end{definition}

 Note that, $\Scal=\varnothing$ if and only if $U$ is strictly convex.  For every $i\in \Ical$ the closure $\overline{S_i}$ is an {\em exposed face} of the convex body $U$ (see, e.g.,  \cite[Definition 2.4.2]{hiriart-urruty_fundamentals_2004}).

\begin{proposition}\label{def:norm_cone_segment}
Let $U\subset\R^2$ be a convex body whose boundary contains a nonempty open segment $S$.   Then the normal cone $N_U(u_S)$ is
constant  for every  $u_S\in S$. Precisely, it turns out to be the ray $N_S:=N_U(u_S)=\{\lambda p_S\ :\ \lambda\ge0\}$ where $p_S\ne0$ is any  normal vector to the segment $S$, exterior to $U$. $N_S$ will be  called the {\em normal cone to the segment $S$}.
\end{proposition}


\begin{proof}
  By the regularity of $S$ the normal cone $N_U(u_S)$ is either the ray $N_S$ or the whole line $L_S=\{\lambda p_S\ :\ \lambda\in\R\}$.
In the latter case, the definition of the normal cone implies that $U$ is contained in a  hyperplane (a line in the case $n=2$) normal to $p_S$. This would imply that the interior of $U$ is empty, contrary to our assumption. Hence the normal cone cannot be the line $L_S$ and the proposition follows. 
\end{proof}


In the rest of this section we adopt the following notation, consistent with Proposition \ref{def:norm_cone_segment}.


\begin{notation}\label{ass:m_segments}
Given  a compact convex body  $U\subset\R^2$, the family  
 $\{S_i\}_{i\in \Ical}$ (possibly empty) of all   disjoint maximal open segments  $S_i=(x_i,y_i)$ contained in $\partial U$ will be denoted by $\Scal$. Moreover,  $N_i$ denotes the normal cone to the segment $S_i$ and  $\Ncal:=\{N_i\}_{i\in \Ical}$ is the family of such cones.
\end{notation}


With a slight abuse of notation  we denote by the same symbol  $\Ncal$ also the {\em union} of the family $\Ncal$ and by $\UNbar$ it closure. 
Let us note that, by maximality of the segments, the normal cones $N_i$ intersect only at the origin, that is, $N_i\cap N_j=\{0\}$ for every $i,j\in \Ical$ with $i\ne j$. 
In the sequel, it is useful to adopt the convention $\UN:=\{0\}$ if the set $U$ is strictly convex and, hence, the family $\{S_i\}_{i\in \Ical}$ is empty; we note that, in such case, $\R^2\setminus\UN=\R^2\setminus\overline{\UN}=\R^2\setminus\{0\}$.

\begin{lemma}\label{lm:geo_conv} Let $U$ be a compact convex body of $\R^n$. We have 
          $\varnothing\ne\partial\sigma_U(p)=\argmax_{u\in U}\inp{p}{u}\subseteq\partial U$ for every $p\in\R^n\setminus\{0\}$. 
        Moreover, if $n=2$ we have: 
          \begin{enumerate}
              \item\label{lem_reg}
           for any  $p\in \R^2\setminus\UN$ the subdifferential $\partial\sigma_U(p)\in \ext(U)$ is single-valued  and {$\sigma_U\in C^1\big(\R^2\setminus\UNbar\big)$;} \item\label{lem_sing} if there exists $i\in \Ical$ such that $p\in N_i\setminus\{0\}$, then $\partial\sigma_U(p)=\overline{S_i}$.
          \end{enumerate}
\end{lemma}


\begin{proof}{\em of Lemma \ref{lm:geo_conv}.} 
Let us note that  the first assertion of the statement can be proven for a general dimension $n\ge2$.  The fact that $\partial\sigma_U(p)=\argmax_{u\in U}\inp{p}{u}$ has already been observed in Remark \ref{Rmk: p nonzero}(a) for every $p\in\R^n$. Let us prove that if $p\ne0$ then it is non-empty and contained in  the boundary. This claim  could be proven by using Bauer's maximum principle (see, e.g., \cite[Lemma 7.69]{aliprantis_infinite_2006}). To be self-contained we prefer here to provide  a complete independent proof.   
    Let $p\in\R^n\setminus\{0\}$. Since $U$ is compact, by Weierstrass theorem, the optimization problem
    $\max_{u\in U}\, \inp{p}{u}$ 
    admits at least one  solution: let us prove that  all of them  belong to  $\partial U$. Assuming by contradiction that a maximizer $\bar u$ belongs to the interior of $U$,  there would exist $\varepsilon>0$ such that $B(\overline{u},\varepsilon)\!\subseteq \,\inter{U}$. Since $p\neq0$, there exists $u\in B(\overline{u},\varepsilon)$ such that $\langle p,u-\overline{u}\rangle> 0$. This contradicts the maximality of $\overline{u}$, indeed
        $\langle p,u\rangle=\langle p,\overline{u}\rangle+\langle p,u-\overline{u}\rangle > \langle p,\overline{u}\rangle$. 
  Then, all maximum points belong to the boundary.

Let us prove {\em(\ref{lem_reg})}. Suppose that  $p\in\R^2\setminus \UN$.
        To prove that $\partial\sigma_U(p)=\argmax_{u\in U}\inp{p}{u}$  is single valued,  
        assume by contradiction that there exist two distinct points $\uo_1, \uo_2 \in \partial U$ that are both maximum points of $u\mapsto \langle p,u\rangle$ on $U$. 
        Since $U$ is convex, the segment $[\uo_1, \uo_2]$ is contained in $U$. 
        By linearity of $p$, also every point on the segment is a maximum point. Since the maximum is attained on  $\partial U$  (already proven), the entire segment must lie on the boundary, i.e., $[\uo_1, \uo_2] \subseteq \partial U$. If $\Scal=\varnothing$ (case of a strictly convex $U$) then we have already reached a contradiction. Otherwise, 
 there exists $i_0\in \Ical$  such that  $(\uo_1,\uo_2)\subseteq S_{i_0}$. Since $\uo_1\ne \uo_2$, then there exists $\uo\in(\uo_1,\uo_2)\subseteq S_{i_0}$ and, by maximality,
        $\inp{p}{\uo}\ge \inp{p}{v}$ for all $v\in U$ 
        or, in other words, $p\in N_U(\uo)$ with $\uo\in S_{i_0}$. Hence, $p\in N_{i_0}$ against our initial assumption. This proves the  uniqueness of the maximum and, thus, the subdifferential turns out to be single-valued.  The claimed fact that it must belong to $\ext(U)$  is a consequence of such uniqueness. Indeed, setting $\ub=\partial\sigma_U(p)$ and assuming by contradiction that it was not an extreme point, there would exist $u_1,\,u_2\in U$ and $\lambda\in(0,1)$ such that $\ub=\lambda u_1+(1-\lambda)u_2$. On the other hand, by maximality and uniqueness we have $\inp{p}{u_1}<\inp{p}{\ub}$ and $\inp{p}{u_2}<\inp{p}{\ub}$, which implies
        $\inp{p}{\ub}=\lambda\inp{p}{ u_1}+(1-\lambda)\inp{p}{u_2}<\inp{p}{\ub}$, 
        that is impossible. 

         It remains to prove that {$\sigma_U\in C^1\big(\R^2\setminus \UNbar\big)$, i.e.,  } that $\partial\sigma_U$ is continuous in the open set  $\R^2\setminus \UNbar$ (where it is single valued) by showing that for every $p_n,p\in\R^2\setminus \UNbar$ we have 
        \begin{equation}\label{eq_contu1seg}
            p_n\to p\quad \implies\quad \partial\sigma_U(p_n)\to \partial\sigma_U(p).
        \end{equation}
       Indeed, since the sequence $\partial\sigma_U(p_n)$ is contained in the compact set $U$, then there exists a subsequence $(p_{n_k})$ of $(p_n)$ and $\ub\in U$ such that $\partial\sigma_U(p_{n_k})\to \ub$. On the other hand, we have $\ub=\partial\sigma_U(p)$; indeed, by maximality of the subdifferential, for every $u\in U$ we have
       $$
      \langle p,\ub\rangle=\lim_{k\to\infty}  \langle p_{n_k},\partial\sigma_U(p_{n_k})\rangle\ge
      \liminf_{k\to\infty}  \langle p_{n_k},u\rangle =\langle p,u\rangle.
       $$
       Since the same argument applies to every subsequence (and $\R^2$ is a metric space) we obtain that the whole sequence $\partial\sigma_U(p_n)$ converges to  $\partial\sigma_U(p)$. 
        This proves \eqref{eq_contu1seg} and, therefore, the claimed continuity of the subdifferential.

         Since for $p\in \R^2\setminus\UNbar$ the subdifferential  $\partial\sigma_U(p)$  is single-valued, 
         then it is a usual gradient (see \cite[Proposition 4.16]{Clarke2013}). Being also continuous, then   $\sigma_U\in C^1\big(\R^2\setminus\UNbar\big)$ as claimed. 
\end{proof}

{\color{black} 
\begin{remark}\label{rem:support_nonC2}
Lemma \ref{lm:geo_conv} extends to general convex sets a well known result holding for strictly convex ones, \cite[Theorem A.1.20]{CSbook} (see also {\cite[Proposition 2.5]{ALS2021}}). The claimed $C^1$-differentiability is independent from the regularity of the boundary that may also have corners or flat points away from segments. On the other hand, simple examples show that $C^2$-regularity fails in presence of corners (Example \ref{ex:convex_sailboat_detailed})  
 or even for smooth 
 strictly convex  sets with flat points on the boundary, like for instance $U=\{(x,y)\in\R^2\ :\ x^4+y^4\le1\}$ (we are indebted to  Prof.\ Lev Lokutsievskiy for this enlightening example). A sufficient condition for $C^k$-regularity of $\sigma_U$ (with $k\ge2$) in the case of a strictly convex set  
 is that it coincides locally with
the epigraph of a function of class $C^k$ with strictly positive definite
Hessian (see \cite[Theorem A.1.21]{CSbook}).  
\end{remark}
}

Recalling that the adjoint state $p(t)$ associated to optimal controls is never zero  (see Lemma \ref{lem_PMP}, assertion (1)),  the previous Lemma \ref{lm:geo_conv} justifies the following definition. 

\begin{definition}\label{def:regimes}
    Let $U$ be a compact convex body of $\R^2$.  
    Let $u$ be an optimal control of $\Pcal_{A,B,U,s}$ and $p$ be a corresponding Pontryagin adjoint state. A non-empty (relatively) open subinterval $I$ of $[0,t_f]$ is called a
    \begin{enumerate}
        \item\label{it:reg_reg} 
     {\em regular regime} if $p(t)\in\R^2\setminus\UNbar$ for every $t\in I$;   \item\label{it:sing_reg}  {\em singular regime} if $p(t)\in \UN$ for every $t\in I$.
    \end{enumerate}   
In case \ref{it:reg_reg}. (respectively, \ref{it:sing_reg}.) we use also to say that a regular (respectively,  singular) regime occurs in the interval $I$.   
\end{definition}

Two open intervals $I$ and $J$ (regimes) are said to be {\em contiguous} if $\overline{I}\cap\overline{J}$ is a singleton. In other terms, two regimes are contiguous if they are disjoint and separated by just one point.

{Suppose that the adjoint curve $p(t)$ be as illustrated in  Figure~\ref{fig:acad_regimes}, where $N_1$ and $N_2$ represent the normals to segments on the boundary of $U$. Although this configuration is not intended to represent a realistic example, it serves to clarify the distinction between regular and singular regimes. Between points $A$ and $B$, the adjoint vector $p(t)$ remains normal to a boundary segment in a relatively open sub-interval $I$ of $[0,t_f]$. Therefore, a singular regime occurs over $I$, or more concisely, as in Definition~\ref{def:regimes}, we say that $I$ is a singular regime. Figure~\ref{fig:acad_regimes} also shows several regular regimes separated by isolated points (namely, $C$, $D$, etc.), corresponding to transitions between contiguous regular regimes.
   \begin{figure}[h!] 
        \centering
        \includegraphics[trim={0pt 0pt 0pt 0pt},clip,width=0.5\textwidth]
        {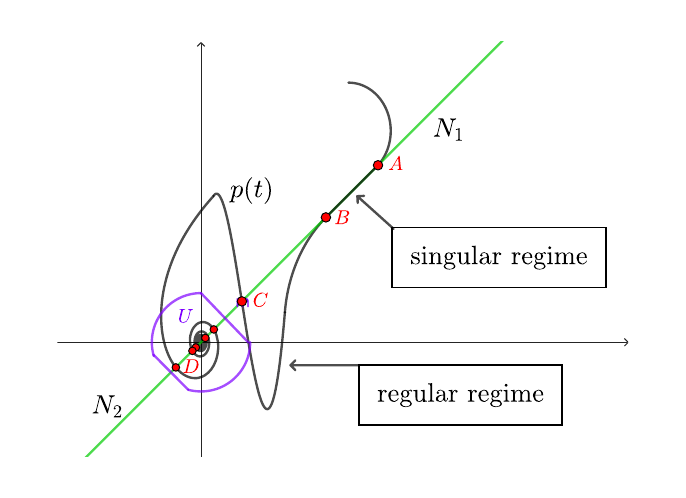}
        \vspace{2ex}
        \caption{A purely illustrative example of regular and singular regimes.}
        \label{fig:acad_regimes}
    \end{figure}
}

The case in which $\#(\Scal)<+\infty$, that is $\partial U$ contains only a finite number of maximal open segments, is special because (the union of the family) $\UN$ is closed and the following theorem holds. 

\begin{theorem}\label{thm:splitting}
 Let $U$ be a compact convex body of $\R^2$, $u$ be an optimal control of $\Pcal_{A,B,U,s}$ and $p$ be a corresponding Pontryagin adjoint state. 
    If $\#(\Scal)<+\infty$, then the following propositions are equivalent:
 \begin{enumerate}
 \item {the topological boundary $\partial p^{-1}(\UN)$ is at most countable,}
     \item    
    the interval   $[0,t_f]$ can be split into a union of at most countably many non-empty subintervals
    whose interior parts  are regular or singular regimes.   \end{enumerate}
\end{theorem}

\begin{proof} As already remarked,  under our hypotheses the set $\UN$ (as  union of the family $\UN$) is closed. {Setting $E=p^{-1}(\UN)$ and $T:=\partial E$, we have
$$
[0,t_f]\setminus T= \inter{E}\cup ([0,t_f]\setminus
E).
$$
Both sets on the right-hand side are relatively open in $[0,t_f]$, and every open subset of $\R$ is the union of at most countably many pairwise disjoint open intervals. The connected components of $\inter{E}$ are precisely the singular regimes, whereas the connected components of $[0,t_f]\setminus E$ are the regular regimes.

If $T$ is at most countable, then, by adding its points as degenerate intervals, we obtain an at most countable decomposition of $[0,t_f]$.

Conversely, suppose that there exists an at most countable decomposition into intervals whose interiors are entirely regular or singular. Then every point of
$T=\partial E$ 
must lie on the boundary of one of the intervals in the decomposition. Since each interval has at most two endpoints, it follows that $T$ is at most countable.}
\end{proof}

{
\begin{remark}\label{rm:splitting_linear}
    Assumption {\it(1)}  is satisfied, for instance, in the case of a position-affine current in which the jacobian $D=\nabla_x s$ is constant, which implies that $p$ is analytic. More generally, one could impose suitable sufficient conditions on $p$ ensuring that {\it(1)} holds, such as quasi-analyticity. A detailed investigation of such conditions, however, goes beyond the scope of the present work.
\end{remark}
}

\subsection{Characterization of the optimal controls}
In this section, using Pontryagin necessary condition, we characterize the  optimal controls in the two different regimes introduced in Definition \ref{def:regimes}. 

{In {\em regular regimes}, {\color{black} slightly strengthening} the regularity of the support function stated by  Lemma~\ref{lm:geo_conv},  {\color{black} we are allowed} to eliminate the adjoint variable $p$ leading to a further optimality necessary condition in the form of a differential equation that must be satisfied by the control function. Its  coefficients depend on  the state $x$ trough the Jacobian of the drift,  and we shall see that, when the control set $U$ is a ball, this equation results in  a non-parametric form of  the celebrated Zermelo Navigation Equation (ZNE). 

It is in fact a general result, highlighted separately in Lemma~\ref{lem:geoZNE}, that the derivation of the ZNE relies solely on the Hamiltonian system and the Weierstrass condition, and thus depends purely on the geometry of the problem.

In addition to regular regimes, where the ZNE holds and the control  has the same  regularity of the current field {\color{black} (provided that the support function be regular enough)}, singular regimes may also arise. In these cases, the optimal control can be discontinuous, and Pontryagin’s necessary conditions leave it largely undetermined.

It is now convenient to introduce the symplectic notation.    Let $J={\small\begin{pmatrix}0&-1\\[0pt]1&0\end{pmatrix}}$ and, for $a,b\in\R^2$, let us set 
$\omega(a,b):=\langle Ja,b\rangle$
the {\em canonical symplectic bilinear form}. It is well known, and easy to check,  that $\omega(a,b)=0$ if and only if the vectors $a$ and $b$ are linearly dependent.   

\begin{lemma}\label{lem:geoZNE} 
Let $U$ be a non-empty bounded subset of $\R^2$  and let $H(t, x, u, p, p_0) :=p_0+\langle p,u+s(t,x)\rangle$. Let $I\subseteq\R$ be open and non-empty, $s\in C^2(I\times\R^2)$ and suppose that $u\in L^0(I,\R^2)$, $x\in W^{1,1}(I,\R^2)$ and $p\in W^{1,1}(I,\R^2)$ satisfy
$$
\begin{cases}
 \dot x(t)=\frac{\partial H}{\partial p}(t, x(t), u(t), p(t), p_0)\\
 \dot p(t)=-\frac{\partial H}{\partial x}(t, x(t), u(t), p(t), p_0)\\
 u(t)\in\partial\sigma_U(p(t))
\end{cases}
$$
for almost every $t\in I$. 
  If   {$\partial\sigma_U$} is single-valued and  twice differentiable  on   an open neighborhood $P$ of $p(I)$, then $u\in C^2(I,\R^2)$ and the following  {\em non-parametric Zermelo navigation equation} holds 
\begin{equation}
    \label{eq:ZNElemma}
\omega\big(\dot u,\ddot u-\nabla_x s(\cdot,x)\, \dot u\big)=0\quad\mbox{ in }I.
\end{equation}
Moreover, if $s\in C^k(I\times\R^2,\R^2)$ and {\color{black}$\sigma_U\in C^{k+1}(P)$}, then  $u\in C^{k}(I,\R^2)$ ($k\in\N\setminus\{0\}$).
\end{lemma}


\begin{proof} Note that the Hamiltonian coincides exactly with the one introduced at the beginning of Section \ref{sec_Pontryagin}. Consequently, the canonical equations match our state and adjoint equations. 

Under our assumptions the subdifferential is a gradient and the inclusion becomes
\begin{equation}\label{eq:u_in_subdiff_sigma_U}
u(t)=\nabla\sigma_U(p(t)).
\end{equation}
Since  $p\in C^1(I,\R^2)$  (see (c) of Remark~\ref{Rmk: p nonzero}), the regularity of  $\sigma_U$ implies that 
the composition on the right hand  side returns a function  $u\in C^1(I,\R^2)$.  Hence, we can differentiate equation \eqref{eq:u_in_subdiff_sigma_U} obtaining
\begin{equation}
    \label{eq:udot=D2}
\dot u(t)=D^2\sigma_U(p(t))\,\dot p(t).
\end{equation}
On the other hand, by the regularity of $u$ and $s$,  the state equation implies $x\in C^1(I,\R^2)$ (in fact, also $C^2))$. Thus, by the adjoint equation, we have $p\in C^2(I,\R^2)$, which (by \eqref{eq:udot=D2}) implies  $u\in C^2(I,\R^2)$. 

Since $\sigma_U$ is $1$-homogeneous, by the Euler's homogeneous function theorem we have {$\inp{\nabla\sigma_U(p)}{p}=\sigma_U(p)$ and, differentiating w.r.t.\ $p$, we get
$D^2\sigma_U(p)\,  p=0$.} By putting all these together,   
we obtain the following important consequence:  
\begin{equation}\label{eq_orthupp_Uconvex}
    \inp{p(t)}{\dot u(t)}=\inp{p(t)}{D^2\sigma_U(p(t))\,\dot p(t)}
    =    0\ \mbox{  for all $t\in I$.} 
    \end{equation}
  Since  $u\in C^2(I,\R^2)$,  we can differentiate also the  equality $\inp{p(t)}{\dot u(t)}=0$ and, using the adjoint equation, we obtain
 \begin{equation}\label{eq_ortho2} 
0=\inp{p(t)}{\ddot u(t)-\nabla_xs(t,x(t))\dot u(t)}
\mbox{  for all $t\in I$.}
\end{equation}
In dimension $n=2$, equations \eqref{eq_orthupp_Uconvex} and   \eqref{eq_ortho2} imply Zermelo navigation equation 
\eqref{eq:ZNElemma}.

To prove the moreover part of the statement 
let us note that, by the state equation, 
if $s\in C^3(I\times\R^2,\R^2)$ then also  
 $x\in C^3(I,\R^2)$.  
 By the adjoint equation then we get $p\in C^3(I,\R^2)$. Under {\color{black}$\sigma_U\in C^4(P)$}, by equation \eqref{eq:udot=D2} we have $u\in C^3(I,\R^2)$ as claimed.  This argument can be iterated to conclude the proof  by induction. 
 
\end{proof}

}

\begin{theorem}
\label{th:char_u_convex-case} Let $U\subset\R^n$ be a compact convex body 
{\color{black}
and 
let  $s\in C^1(\R_+\times\R^n,\R^n)$.}
If {\color{black} $(u,x)$ is an optimal solution} for problem $\Pcal_{A,B,U,s}$, then $u(t)\in\partial U$ for a.e.\ $t\in[0,t_f]$. 
    Moreover, if $n=2$  and {\color{black} $p$ is a Pontryagin adjoint state}, 
    then  
    we have  
    \begin{enumerate}
    \item\label{th:regular}{\em(regular regimes)}  {\color{black}if $I$ is a regular regime then $u\in C(I,\ext(U))$ and $x\in C^1(I,\R^2)$; if, moreover,   $\sigma_U\in C^3(P)$ where $P$ is an open neighborhood of $p(I)$, then 
    $u\in C^2(I,\ext(U))$}   
  and satisfies  the following conditions:
 
         \begin{enumerate}
         \item  {\em (non-parametric) Zermelo navigation equation } 
            \label{it:ZNE1} 
            \begin{equation}
            \label{eq_ZNE1_Uconvex} 
                \omega\big(\dot u,\ddot u-\nabla_x s(\cdot,x)\, \dot u\big)=0{\color{black}\quad\mbox{ in }I};
            \end{equation}
              \item\label{it:ZNE0}  if, moreover, $s=s(x)$ (i.e., $s$ is independent of $t$), then 
        either 
            $\omega(\dot u,{u+s(x)})\equiv0$ {\color{black}in the abnormal case ($p_0=0$)},   
         or $\omega(\dot u,{u+s(x)})$ is never  zero {\color{black}in the normal one ($p_0=-1$)};
         \item if  {\color{black}$s\in C^k(I\times\R^2,\R^2)$ and $\sigma_U\in C^{k+1}(P)$, then  $u\in C^{k}(I,\ext(U))$} ($k\in\N$, $k\ge2$);
    \end{enumerate}
    \item\label{item_charth_sing} {\em(singular regime)} if {\color{black}$I$} is a singular regime then there exists 
    {\color{black}$i\in \Ical$} 
     such that 
     $u(t)\in {\color{black}\overline{S_{i}}}$ for a.a.\ {\color{black}$t\in I$}.
    \end{enumerate}
\end{theorem}

{\color{black}
 \begin{remark}
 The assumption $\sigma_U\in C^3(P)$ provides sufficient regularity to derive Zermelo equation according to Lemma \ref{lem:geoZNE}   (see also Lemma \ref{lm:geo_conv} and Remark \ref{rem:support_nonC2}).
  \end{remark}
}

{
\begin{remark} (a)
In the case $n=2$, decomposing $\nabla_x s = E + W$ into its symmetric and antisymmetric parts, with
\[
E=\tfrac12(\nabla_x s+\nabla_x s^\top),
\qquad
W=\tfrac12(\nabla_x s-\nabla_x s^\top)=\alpha J,
\]
where $\alpha=\tfrac12(\partial_{x_1}s_2-\partial_{x_2}s_1)$, we obtain
\(
\omega(\dot u,\ddot u)
=
\omega(\dot u,E\dot u)
+
\alpha\,|\dot u|^2.
\)
Hence the signed (oriented) curvature of the curve $u(t)$,
\(
\kappa_{\mathrm{or}}(u)
:=
{\omega(\dot u,\ddot u)}/{{|\dot u|^3}}, 
\)
satisfies
\[
\kappa_{\mathrm{or}}(u)
=
\frac{\omega(\dot u,E\dot u)}{{|\dot u|^3}}
+\frac{\alpha(x)}{{|\dot u|}}.
\]
Thus, the curvature of the optimal control  is influenced both by the vorticity
$\alpha=\tfrac12\operatorname{curl}s$ of the current and by its symmetric (strain) component through the term $\omega(\dot u,E\dot u)$.

   (b) In components, equation \eqref{eq_ZNE1_Uconvex}   writes 
    
    \begin{equation}
    \label{eq_ZNE1_component}  
       \ddot u_1\dot u_2-\dot u_1\ddot u_2+\dot u_1^2s_{2,1}(\cdot,x)-\dot u_1\dot u_2\big(s_{1,1}(\cdot,x)-s_{2,2}(\cdot,x)\big)-\dot u_2^2s_{1,2}(\cdot,x)=0,
    \end{equation}
    where  $s_{i,j}:=\frac{\partial s_i}{\partial x_j}$. 
\end{remark}}

\begin{proof}{\em of Theorem \ref{th:char_u_convex-case}.} 
 Under our assumptions, the PMP (Lemma \ref{lem_PMP}) holds and $u$ satisfies, in particular, Weierstrass condition 
 which 
 can also be written in the form (see (a) of Remark~\ref{Rmk: p nonzero})
$u(t)\in\partial\sigma_U(p(t))$ 
    for almost all $t \in [0, t_f]$. By   Lemma \ref{lm:geo_conv}, the subdifferential $\partial\sigma_U$ 
takes values on $\partial U$. 
Then we have $u(t)\in\partial U$ for a.a.\  $t\in[0,t_f]$, as claimed.  

Let us prove {\em1.} In regular regimes, 
  by Lemma \ref{lm:geo_conv} again, the subdifferential $\partial\sigma_U$ 
is single valued, it belongs to $\ext(U)$ {\color{black}and is continuous, as well as  $p$. By Weierstrass condition $u(t)=\partial\sigma_U(p(t))$ it immediately follows that $u$ is continuous. By the state equation and the regularity of $s$ we obtain $x\in C^1(I,\R^2)$ and the first part of {\em(1)} is proved.   

Under the additional  assumption that $\sigma_U$ be
  of class  $C^3$  in  an open neighborhood $P$ of $p(I)$}, Lemma \ref{lem:geoZNE} applies and we get   {\color{black}$u\in C^2(I,\R^2)$,} Zermelo equation \eqref{eq_ZNE1_Uconvex} holds, and also part {\em1.(c)} turns out to be proven.

Assertion {\em1.(b)} follows by proving that, when $s$ is independent on $t$, only the following  cases can occur:
    \begin{enumerate}
        \item[(i)] $p_0=0$ $\iff$ $\omega(\dot u(t),{u(t)+s(x(t))})=0$ for every $t\in R$, 
    \item[(ii)] $p_0=-1$ $\iff$  $\omega(\dot u(t),{u(t)+s(x(t))})\ne0$ for every $t\in R$.
\end{enumerate}
    Indeed, by \eqref{eq_ch}, $p_0=0$ iff $\inp{p}{u+s(x)}=0$ a.e.\ on $(0,t_f)$ iff (by \eqref{eq_orthupp_Uconvex} and the regularity of $u$, $s$ and $x$) $\omega(\dot u,{u+s(x)})\equiv0$ on {\color{black}$I$}, which proves {(i).}
    On the other hand, by the same arguments and by continuity, we have that  $p_0=0$ iff $\exists$ {\color{black}$t_0\in I$} such that $\inp{p(t_0)}{u(t_0)+s(x(t_0))}=0$ iff $\exists$ {\color{black}$t_0\in I$} such that $\omega(\dot u(t_0),u(t_0)+s(x(t_0)))=0$,  which proves {(ii)}\ because $p_0$ is constant and can take only the values $0$ and $-1$.  

   Let us prove assertion {\em2}. If $p(t)\in \cup_{i\in \Ical} N_i$ in an interval {\color{black}$I$}, by continuity and the fact that $p$ is never zero (see Lemma \ref{lem_PMP}), there exists 
    {\color{black}$i\in \Ical$}  such that 
    $p(t)\in N_{i_0}$ for every  {\color{black}$t\in I$.} Therefore, by \eqref{eq:u_in_subdiff_sigma_U} and {\em(2)} of Lemma \ref{lm:geo_conv}, we have 
     {\color{black}$u(t)\in \overline{S_{i}}$} for a.e.\  {\color{black}$t\in I$,} which proves assertion {\em(2)} of the statement. 
     The theorem is, thus, completely proven. 
\end{proof}

\begin{remark}{If a singular regime {\color{black}$I$} (i.e., scenario {\em\ref{item_charth_sing}})   of Theorem \ref{th:char_u_convex-case}) occurs,  then in the time interval {\color{black}$I$} the control $u$ may be discontinuous. Nevertheless, discontinuities may appear also between  contiguous regular regimes, even if singular ones do not occur (see Example \ref{ex:convex_sailboat_detailed}).   Such discontinuities would correspond to  tacking maneuvers, very usual in sailing navigation.}          
\end{remark}

\subsection{Existence of singular regimes: necessary condition}\label{subsec:SingularRegimes}

From Definition \ref{def:regimes}, it trivially follows that singular regimes cannot occur when $U$ is strictly convex (i.e., $\UN=\varnothing$). 
In this subsection we present another useful necessary condition for the existence of  a singular regime in the non-trivial case of a non strictly convex control set $U$.  
This condition involves the first-order structure of the current field (specifically, the Jacobian matrix {$\nabla_x s$} and its eigenvectors)  along the optimal trajectory. Given a $2\times2$ square matrix $D$, let us denote by $E_\lambda$ the eigenspace  corresponding to an eigenvalue $\lambda$ in the spectrum $\sigma(D)$ of $D$ (with the convention $E_\lambda(D)=\{0\}$ if $\lambda$ is not purely real)
and 
$E(D):=\cup_{\lambda\in\sigma(D)}E_\lambda(D)$.  

{By definition, a singular regime occurs when the adjoint function $p(t)$ evolves along the exterior normal to a segment on the boundary of $U$. By the adjoint equation, this means that the corresponding exterior normal vector is an eigenvector of the Jacobian  $D(t):={\nabla_x} s(t,x(t))^\top$. This is stated and proved in the following theorem that gives a  necessary condition for the existence of a singular regime.}

\begin{theorem}[necessary condition for the existence of a singular regime]   \label{th:sing_reg_NC} Let $U$ be a compact convex body, {$s\in C^1(\R_+\times\R^2,\R^2)$} and $u$ be an optimal control for problem $\Pcal_{A,B,U,s}$. 
If a singular regime occurs in a time interval $I_0$, then 
$$
E(\nabla_x s(t,x(t))^\top)\cap \UN\ne\{0\}\quad \forall\,t\in I_0.
$$
In other words, there exists $i_0\in \Ical$ such that, for every $t\in I_0$, the exterior normal vector  $p_{i_0}$ to the segment $S_{i_0}$ is an eigenvector of $D(t):={\nabla_x} s(t,x(t))^\top$. 
\end{theorem}

\begin{proof} 
If  a singular regimes occurs in the interval $I_0$, by Definition \ref{def:regimes}, we have that $p(t)\in\UN$ for every $t\in I_0$.
Since $p\in C^1$ (see Remark \ref{Rmk: p nonzero}(b)), the normal cones share only the origin and $p(t)$ is never zero, then  
 there exists $i_0\in \Ical$, 
$p_{i_0}\in N_{i_0} \setminus\{0\}$ with $u_{i_0}\in S_{i_0}$,  and  $\mu\in C^1(I_0,(0,\infty))$  such that  $p(t)=\mu(t) p_{i_0}$. 
Then, the adjoint system \eqref{eq:AdjointEquation}  gives
\begin{equation}\label{eq:curr_lambda}
\dot \mu(t)\,p_{i_0}=-\mu(t)D(t) p_{i_0}.
\end{equation}
Since, {as already observed,} $p_{i_0}\ne0$,  $p(t)\ne0$, {and $\mu(t)\ne0$ for every $t\in I_0$, then we get 
\begin{equation*}\label{eq:eigen_mu}
D(t) p_{i_0}=-\dfrac{\dot \mu(t)}{\mu(t)}\,p_{i_0},
\end{equation*}
that is, $p_{i_0}$  is an eigenvector of $D(t)$ corresponding to the eigenvalue $\lambda(t)=-\dfrac{\dot \mu(t)}{\mu(t)}$, as claimed.
}
\end{proof}

{
\begin{remark}
After Theorem \ref{th:sing_reg_NC}, we see that, in general, the presence of a segment on the boundary of $U$ does not by itself guarantee the existence of a singular regime, since the Jacobian matrix $D={\nabla_x} s^\top$ may fail to admit eigenvectors orthogonal to that segment. The following remark better clarifies this phenomenon.
\end{remark}
}

\begin{remark}
 (a) The necessary condition of Theorem \ref{th:sing_reg_NC} cannot be satisfied (and singular regimes can never occur) whenever the matrix $D={\nabla_x} s^\top$ is constant (e.g.,  $s(t,x)=Dx+s_0(t)$) and has no real eigenvalues (and regardless of the choice of $U$).      This happens, for instance, when  {$s(t,x)=\delta(-x_2,x_1)+s_0(t)$,}
 for suitable   $\delta\in\R$ and $s_0\in C^1$.

\label{rem:NC_D=0}
    (b) The necessary condition of Theorem \ref{th:sing_reg_NC} is always satisfied when $\nabla_x s\equiv0$  (case of a current independent of $x$ but possibly depending on $t$).
Nevertheless,  for a wide choice of current fields and control sets $U$ the necessary condition is not satisfied even if real eigenvalues exist (see Example \ref{ex:convex_sailboat_detailed})  and, thus,  singular regimes are not allowed.
Conversely, in some circumstances a singular regime not only can, but must, occur (see Remark  \ref{rem:Cubar_geo} and Example~\ref{ex:singular_must_occur}). 

(c) The case in which singular regimes can be excluded (e.g., when the necessary condition is not satisfied) are particularly important when the set of admissible velocities $V$ is non-convex. Indeed,  in such case one is lead to consider the convexified problem $P_{\co(V)}$ which, in the absence of singular regimes, solves also the non-convex problem (see the forthcoming  Section \ref{sec:nonconvex}).    

(d) It should be emphasized that the singular phenomena described in this paper are not driven by Lie bracket effects in the classical sense of low-dimensional single-input systems, as described in~\cite{Sussmann1987_1,KS1989,BC2003,
 BCW2023}. Indeed,
the controlled directions already span $\mathbb{R}^2$, so that the control vector fields form a full-rank distribution whenever $0 \in \operatorname{int}(U)$. Hence the Lie algebra does not generate additional directions. 
\end{remark}

\subsection{Normal cones and invariant sets} By Weierstrass necessary condition $p(t)\in N_U(u(t))$ for a.e.\ $t\in(0,t_f)$,  the optimal control $u(t)$ is characterized by the evolution of the adjoint state $p(t)$. 
On the other hand, $p(t)$ solves a linear differential equation
(the adjoint equation \eqref{eq:AdjointEquation}).
The study of invariant sets of the adjoint system, in connection with the structure of the normal cones to $U$,  is helpful in characterizing the optimal controls, as stated by the following two lemmas. Actually, these results will play a crucial role in the synthesis of optimal controls for a   linear current (see Example \ref{ex:convex_sailboat_detailed}).

\begin{lemma}\label{lem:NUcovering}
Given a compact convex body  $U\subset\R^2$, the family of normal cones 
$\{N_U(u)\}_{u\in\ext(U)}$
covers the plane. Moreover, the intersection of any pair of distinct cones of the family is either $\{0\}$  or a ray normal to a segment of the boundary. 
\end{lemma}
\begin{proof} The first part of the statement is well known also in dimension greater than 2 (see, e.g., \cite[Equation (2.2.5)]{Schneider1993}).   
Let us prove the part regarding the intersection of the normal cones, in dimension $2$.   

Let $x,y \in \ext(U)$, $x \neq y$, and assume $N_U(x) \cap N_U(y) \neq \{0\}$. 
Let $p \neq 0$ belong to this intersection.
Then, by definition of  normal cone to a convex set (see \eqref{def_normal_cone}),
\begin{equation}\label{eq:NUxy}
\langle p, u-x \rangle \le 0 \quad\text{and}\quad
\langle p, u-y \rangle \le 0
\quad\forall u \in U.
\end{equation}
Equivalently, in terms of the support function $\sigma_U$ of $U$, 
\(
\langle p, x \rangle = \langle p, y \rangle = \sigma_U(p),
\)
so both $x$ and $y$ belong to the subdifferential 
$\partial\sigma_U(p) := \{ v \in U : \langle p, v \rangle = \sigma_U(p) \}$. 
Since $U \subset \mathbb{R}^2$ is compact and convex, $\partial\sigma_U(p)$ is either a point or a line segment on the boundary  (see Lemma~\ref{lm:geo_conv}). 
Since $x \neq y$, it must be a nondegenerate segment with endpoints $x$ and $y$, that is, $\partial\sigma_U(p)=[x,y]$. As already observed (Proposition \ref{def:norm_cone_segment}), for every $s\in(x,y)$, 
$N_U(s)$ lies on the normal line to the segment. Moreover, we have 
$p\in N_U(s)$. Indeed, 
there exists $\lambda\in(0,1)$ such that $s=\lambda x+(1-\lambda)y$ and, using  \eqref{eq:NUxy}, for every $u\in U$ we have
$$
\langle p,u-s\rangle=
\langle p,\lambda u+(1-\lambda)u-[\lambda x+(1-\lambda)y]\rangle
=\lambda \langle p,u-x\rangle+(1-\lambda)\langle p,u-y\rangle\le0,
$$
that is $p\in N_U(s)$, by definition. 
Therefore,
\(
N_U(x) \cap N_U(y)\subseteq N_U(s).
\)

To prove the converse inclusion we take $p\in N_U(s)$ and observe that, for every $u\in U$
$$
\langle p,u-x\rangle=\langle p,u-s\rangle+\langle p,s-x\rangle\le0
$$
because $p\in N_U(s)$ and $N_U(s)$ is normal to $x-s$. Thus, $p\in N_U(x)$. 
Analogously, $p\in N_U(y)$ and the lemma is completely proven.   
\end{proof}

\begin{lemma}\label{lem:invariance}
Let $U$ be a compact convex body of $\R^2$ and $s\in C^1(\R_+\times\R^2,\R^2)$. 
Suppose that $(u,x)$ is an optimal solution of $\Pcal_{A,B,U,s}$ and that $p$ is a corresponding adjoint state. Suppose that there exists a cone $Q$ such that $p(t)\in Q$ for every $t$ in an open interval $I\subseteq[0,t_f]$ and that the following property holds
\begin{equation}\label{eq:ass_UQ}
u\in U\setminus Q\ \implies\ N_U(u)\cap Q=\{0\}.
\end{equation}
Then $u(t)\in Q\cap\partial U
$ for a.e.\  $t\in I$.
If, moreover, $I$ is a regular regime, then $u(t)\in Q\cap \ext(U) $ 
for a.e.\  $t\in I$.
\end{lemma}
\begin{proof}
Since $p(t)\in Q$ and (by   nondegeneration and Weierstrass  conditions) $0\ne p(t)\in N_U(u(t))$ for a.e.\ $t\in(0,t_f)$, then 
$p(t)\in N_U(u(t))\cap Q$, and hence $N_U(u(t))\cap Q\ne\{0\}$,  for a.e.\ $t\in I$. This, by assumption \eqref{eq:ass_UQ}, implies $u(t)\in Q$ for a.e.\ $t\in I$. The conclusion follows by Theorem~\ref{th:char_u_convex-case} (see also Lemma \ref{lm:geo_conv}).
\end{proof}

\section{Parametrization of the boundary and classical ZNE}\label{sec_paraboundary}

{Theorem \ref{th:char_u_convex-case} states, among other things, that the optimal control takes values on $\partial U$ and that the navigation equation can be obtained without imposing any parametrization of the boundary. Thus, in regular regimes, Zermelo’s equation \eqref{eq_ZNE1_Uconvex} is inherently ``\emph{non-parametric}''. On the other hand, the availability of a regular parametrization of (the relevant portion of) $\partial U$ allows one to rewrite the equation, at least locally,  in a ``\emph{parametric}'' form. This representation provides a particularly effective tool for understanding how the geometry of $U$ enters the navigation dynamics and, notably, for establishing a direct and transparent comparison with the classical Zermelo navigation equation, which corresponds to the case $U=\mathbb{B}(0, V)$ with $V>0$. Through this comparison, one can clearly identify which terms are dictated by the general geometry of the control set $U$ and  which are specific to the case where $U$ is a Euclidean ball. The aim of this short section is  to develop this parametric formulation, highlight the resulting geometric insights, and establish a direct link with Zermelo's original findings~\cite{ZermeloII}.}

Let $\gamma:\R\to\partial U$ and $\alpha:[0,t_f]\to\R$ be $C^2$ functions such that  $u(t)=\gamma(\alpha(t))$ for every $t$ in an (relatively) open interval $I\subseteq[0,t_f]$.  Heuristically, the function $\gamma$ takes into account the geometry of the boundary, while $\alpha$ is typically an angle that can have a physical meaning like the heading angle, or a geometrical meaning as, for instance, the arclength parameter (or even both ones).

Computing the derivatives of $u$ and substituting 
in the Zermelo equation \eqref{eq_ZNE1_Uconvex} (or~\eqref{eq_ZNE1_component}), for every $t\in I$ 
we  have $\dot\alpha(t)=0$ or 
\begin{equation}
\label{eq_ZNE_gamma}
\dot\alpha(t)d_\kappa(\alpha(t)) = \det\big(\gamma'(\alpha(t))|\nabla_x s(t,x(t))\gamma'(\alpha(t)\big),
\end{equation}
where $d_k:=\gamma_2''\gamma_1'-\gamma_1''\gamma_2'$ is the determinant that appears in the numerator of the expression for the curvature of $\gamma$. {If, moreover,  $s$ is independent of $t$, then 
        either 
            $\dot\alpha(t)\det\big({\gamma'}(\alpha(t))|\gamma(\alpha(t))+s(x(t))\big)
            \equiv0$,   
         or it is never equal to zero (normal case).}

The   {\em parametric Zermelo navigation equation \eqref{eq_ZNE_gamma}}  characterizes the optimal control $u(t)=\gamma(\alpha(t))$ 
whenever one can exclude that $\dot\alpha=0$ such as, for instance, {in the {\em normal case} of a $t$-independent current mentioned above.}

\begin{example}[the classical ZNE]
\label{ex_ZNEclassic}{When $U$ is a ball of radius $V>0$, its boundary admits the parametric arclength  representation $\gamma_1(\theta)=V\cos\theta$, $\gamma_2(\theta)=V\sin\theta$, where $\theta$ (playing the role of the parameter $\alpha$ considered above) is the angle with respect to the positive direction of the abscissa. We have  {$\gamma'(\theta(t))=\big(-V\sin(\theta(t)), V\cos(\theta(t))\big)^\top$  and $\gamma''(\theta(t))=\big(-V\cos(\theta(t)), -V\sin(\theta(t))\big)^\top$ }   which implies that $d_\kappa(\theta){=V^2}$ {for all $\theta$}. {Then, in this case, equation \eqref{eq_ZNE_gamma}  turns  out to be}
\begin{equation*}
\dot\theta(t)=s_{2,1}(t,x(t))\sin^2\theta(t)+\big(s_{1,1}(t,x(t))-s_{2,2}(t,x(t))\big)\sin\theta(t)\cos\theta(t)-s_{1,2}(t,x(t))\cos^2\theta(t),\label{eq_ZNEclassic}
\end{equation*}
which is the {\em classical Zermelo navigation equation}   \cite{ZermeloII} (see also \cite{Mania1937}). 
}
{We have thus shown that the classical Zermelo navigation equation is recovered as a particular case of our general framework, and corresponds to the choice of a Euclidean ball as control set $U$ and of the regular angular parametrization of its boundary. However, the parametric formulation~\eqref{eq_ZNE_gamma} applies to much more general geometries of the control set $U$. In the subsequent examples, we illustrate how this broader perspective allows one to describe navigation dynamics beyond the classical setting.}
\end{example}

\begin{example}[polar coordinates with   $V=V(\theta)$]\label{ex_ZNEestended}{A more general case, compared to that of Example~\ref{ex_ZNEclassic}, can be handled by allowing $V$ to depend on the angle $\theta$ in the expressions of $\gamma$. As before, the curve $\theta\mapsto \gamma(\theta)$ represents a parametrization of the boundary of $U$ which now is no longer constrained to be a sphere. 
    Let us assume that $V$ be regular enough, namely   $V\in C^2$ and regular in the sense of polar curves, that is $V(\theta)^2+V'(\theta)^2>0$ for every $\theta$. 
With the same notation introduced above, we have 
$$
        d_\kappa(\theta)=V(\theta)^2+2V'(\theta)^2-V(\theta)V''(\theta).
    $$
    As before, we assume  $s\in C^2$.
    Computing  the derivatives of $u$ and substituting in the parametric  Zermelo equation \eqref{eq_ZNE_gamma} we
    have
    \begin{equation}
         \label{eq_ZNEgen}
        \begin{aligned}
    \dot\theta d_\kappa(\theta)=&\,
        s_{2,1}(\cdot,x)\big(V'(\theta)\cos\theta-V(\theta)\sin\theta\big)^2
            -s_{1,2}(\cdot,x)\big(V'(\theta)\sin\theta+V(\theta)\cos\theta\big)^2\\
            &+\big(s_{2,2}(\cdot,x)-s_{1,1}(\cdot,x)\big)\big(V'(\theta)\sin\theta+V(\theta)\cos\theta\big)\big(V'(\theta)\cos\theta-V(\theta)\sin\theta\big).
        \end{aligned}
         \end{equation}
    For $V(\theta)=V$ constant,  equation \eqref{eq_ZNEgen} reduces to the classical Zermelo   equation of  Example \ref{ex_ZNEclassic}. 
}
\end{example}

\section{
Navigation equations for a strictly convex set
}
\label{sec_U_strictly_convex}

The particular case of a strictly convex control set $U$ is characterized by a single global regular regime and, thus,  Theorem \ref{th:char_u_convex-case} admits the following corollary.

\begin{corollary}
\label{cor:strictly_convex} 
Let $U$ be a compact and strictly-convex body of $\R^n$ {\color{black} and 
$s\in C^1(\R_+\times\R^n,\R^n)$.  If $(u,x)$ if an optimal solution for problem $\Pcal_{A,B,U,s}$
then $u\in C((0,t_f),\partial(U))$ and $x\in C^1((0,t_f),\R^n)$. 
Moreover, if $n=2$, 
$s\in C^k(\R_+\times\R^2,\R^2)$  for some $k\ge2$, $p$ be a Pontryagin adjoint state and $\sigma_U\in C^{k+1}(P\setminus\{0\})$
in an open neighborhood $P$ of $p(0,t_f)$, 
  then $u\in C^{k}((0,t_f),\partial U)$} and satisfies on $(0,t_f)$ the following conditions 
    \begin{enumerate}
          \item[(a)] {\em (Zermelo navigation equation)} 
            \begin{equation}
            \label{eq_ZNE1} 
                \omega(\dot u,{\ddot u-\nabla_x s(\cdot,x)\, \dot u})=0.
            \end{equation}
             \item[(b)] if, moreover, $s=s(x)$ (i.e., $s$ is independent of $t$), then 
        either 
            $\omega(\dot u,{u+s(x)})\equiv0$ {\color{black}in the abnormal case ($p_0=0$)},   
         or $\omega(\dot u,{u+s(x)})$ is never  zero {\color{black}in the normal one ($p_0=-1$)}.
    \end{enumerate}
\end{corollary}

\begin{remark}\label{rem:ZNEstrconvex} (a) It is clear that  $u$ is not unique, in general. Indeed, by taking, for example, $U$ to be a ball centered in the origin, a point $B\ne0$ and  $A$ a half sphere centered in $B$, we have that from any  starting point in $A$ it takes the same time to reach $B$.  

{\color{black}
(b) As already observed (Remark \ref{rem:support_nonC2}), the condition $\sigma_U\in C^{k+1}(P\setminus\{0\})$ is satisfied by strictly-convex and compact sets $U$ that coincide locally with
the epigraph of a function of class $C^k$ with strictly positive definite
hessian. Hence, for such sets Zermelo equation globally holds. 

(c)  Corollary  \ref{cor:strictly_convex} recovers} the celebrated Zermelo's navigation equation in the case in which $U$ is a closed ball (see Example \ref{ex_ZNEclassic}). 
In fact, in his problem proposed and solved in 1931  (\cite{ZermeloII})  Zermelo considered the particular case in which the vessel can move at maximum speed $V$ in any direction and at any point.
In solving his problem, Zermelo could not have used Pontryagin's principle, which dates from 1956.

(d) In the setting of Example \ref{ex_ZNEestended}   
     the strict convexity of $U$ would correspond
    to strict positivity of   the curvature 
    $$
        \kappa(\theta)=\frac{\big|V(\theta)^2+2V'(\theta)^2-V(\theta)V''(\theta)\big|}{\big(V(\theta)^2+V'(\theta)^2\big)^{3/2}}
    $$
    of the polar curve $V=V(\theta)$, which implies  
    $    d_\kappa(\theta)=V(\theta)^2+2V'(\theta)^2-V(\theta)V''(\theta)\ne0.$ 
    In such case, equation \eqref{eq_ZNEgen} admits the normal form $\dot\theta(t)=f(t,x(t),\theta(t))$.
\end{remark}

\section{The case of an affine current}\label{sec:examples}

 In this section we present the particular case of planar navigation  in presence of a position-affine current $s(t,x)=s_0(t)+D^\top\!(t)x$ with $s_0\in C^1(\R_+,\R^2)$ and $D\in C^1(\R_+,\R^{2\times 2})$. Moreover, for simplicity, and because it  is the most interesting scenario in applications, we restrict ourselves to a (convex and compact) control set $U$ having a finite number $m\in\N$ of  segments on the boundary (that is,  $\Ical=\{1,\dots,m\}$ with  the notation introduced in Section~\ref{ass:m_segments}). This includes the case of a strictly convex set $U$ ($m=0$).

When the current is an affine function of the position, 
the geometric necessary condition for the existence of a singular regime, 
Theorem \ref{th:sing_reg_NC},  is particularly beneficial because ${\nabla_x} s=D^\top$  does not involve the state trajectory $x(t)$. 
Moreover  (due to the constancy of $U$), when  the eigenspaces of
$D$ are constants, if the necessary condition is satisfied on an interval $I$ then it  holds  
on the whole interval $[0,t_f]$. In fact, the following stronger property
can be proved.

\begin{theorem}\label{th:affine_regimes}
In the case of a position-affine current of the form $s(t,x)=D(t)x+s_0(t)$ with $D(t)=\delta(t)D_0$ for $D_0\in\R^{2\times 2}$  constant and $\delta\in C^1(\R_+,\R\setminus\{0\})$, if $u$ is a time-optimal control for problem $\Pcal_{A,B,U,s}$ then 
we have:
\begin{enumerate}
    \item if there exists a singular regime, then it is global;
    \item if $\delta$ is constant, then 
the following alternative holds: 
\begin{enumerate}
    \item there exists a global singular regime
    \item
the interval $[0,t_f]$ is union of (the closure of)
 a finite number of regular regimes; they reduce to, at most, two of them if   
    the eigenvalues of $D_0$ (equivalently, of $D$) are real.
    \end{enumerate}
\end{enumerate}
\end{theorem}


\begin{proof}   Let us prove {\em(1)}. First of all we note that $E(D(t))=E(D_0)$ for every $t$ and the necessary condition of Theorem \ref{th:sing_reg_NC} reads 
$E(D_0)\cap\Ncal\ne\{0\}$.  
Hence, the occurrence of a singular regime on an interval $I$ implies the existence of a real eigenvalue $\lambda$ of $D_0$ and an index $i\in\Ical$ such that $p(t) \in N_i \cap E_{\lambda}(D_0)$ for all $t \in I$. The claim follows by observing that $N_i \cap E_\lambda(D_0)$ is an (forward and backward) invariant  cone for the adjoint system  $p' = -D(t) p$. Indeed, starting from $p(t_0)\in N_i \cap E_\lambda(D_0)$  for $t_0\in I$, it is immediately seen that  $p(t)= e^{-\lambda\int_{t_0}
^t\delta(\xi)\,d\xi}p(t_0)\in  N_i \cap E_\lambda(D_0)$ for every $t\in[0,t_f]$.

Let us prove {\em(2)}. Suppose that a singular regime does not occur. By Theorem \ref{thm:splitting}  (see also Remark \ref{rm:splitting_linear}) the interval $[0,t_f]$ splits into an (at most countable)  union of regular regimes. By Theorem \ref{th:char_u_convex-case}, each separation point between them corresponds to an instant $t$ in which  $p(t)\in N_i$ for some $i\in\Ical$. Such instants $t$ must be necessarily isolated, because singular regimes are excluded. Thus, the claim can be proven by showing that $p(t)$ can intersect the normal rays $N_i$ only a finite number of times (at most one if the eigenvalues are real). 

On the other hand, it is well known (see \cite[Chapter 1] {perko_differential_2009}) that, 
on a finite time-horizon, any non-zero trajectory in the phase diagram of a $2\times2$ linear system  $y'=Ay$ can intercept a ray exiting from the origin and not aligned with an eigenvector of $A$ only a finite number of times. Such intersection points reduce to  at most  one if the eigenvalues of $A$ are real. 

Thus, the claim follows by taking $A=-D$, in the case in which the ray $N_i$ is not aligned with an eigenvector. 
Otherwise, since the 
corresponding eigenvalue would be necessarily real, if $p(t)$  intersects $N_i$ then (arguing as in the proof of part {\em(1)}) it must belong to $N_i$ for every $t\in[0,t_f]$. This would be a singular regime, excluded by our assumptions. Hence, when $N_i$ is aligned with an eigenvector, the adjoint state  $p(t)$ cannot intersects $N_i$ and the theorem is completely proved.  
\end{proof}

The case of a $t$-independent affine current can be considered as a first approximation of a more complex one.  This case can be reduced to that of a linear  current $s(x)=D^\top x$ on the (constant) control set $$U_0:=U+s_0$$ which, on the other hand, preserves the same properties (compactness, convexity, non empty interior) and structure of the boundary (finite number of boundary segments)  of the set $U$. 
For this reason,  from now on we split the study in the two subcases of a constant and a linear current.

 \subsection{The case of a constant current}\label{subsec:constant_current}

 Suppose that the current is  constant, i.e.,  $s(t,x)=s_0\in\R^2$ for every $x$ and that $u$ is an optimal control.  Then $\nabla_x s=0$ and the adjoint system gives $p(t)=p$  constant. Thus, in such case there exists a global  regime, regular or singular (according to Theorem~\ref{th:affine_regimes}).   Moreover, we can set the current to be $0$ by working with the set $U_0$ in place of $U$. Having this in mind,  in the following, we continue to denote the control set by $U$.     

The main result of this subsection is the following theorem that completely characterizes  
the optimal control when the starting and target regions, $A$ and $B$,  reduce to single points. 
In the statement we will make use of the following notations. 
Given two distinct points $P,Q\in\R^2$, let us denote by 
$$
[P,Q\rangle:=\{P+\lambda(Q-P)\ :\ \lambda\ge0\}
$$
the {\em ray originating from 
$P$ and passing through $Q$.} With this notation,  and $O=(0,0)$, we have 
$$
[O,Q-P\rangle=\{\lambda(Q-P)\ :\ \lambda\ge0\},
$$
that is the positive closed cone generated by the vector $Q-P$. 

\begin{theorem}\label{th:char_ABpoints_s=const} Suppose that $A$ and $B$ are singletons and $U$ be a compact convex body of $\R^2$ with $m$ segments in its boundary. 
If an optimal control for problem $\Pcal_{A,B,U,s_0}$ exists, then the (unique) constant control 
$$
\ub=\argmax\{|u|\ :\ u\in[0,B-A\rangle\cap\partial U\}
$$
is optimal. If $\ub\in S_i$ for some $i\in\{1,...,m\}$ and the support line of $S_i$ does not contain the origin 
then every measurable selection $u(t)\in \overline{S_i}$ such that $x^u(t)=B$ for some $t\geq 0$  is optimal as well. Otherwise, $\ub$ is the unique optimal control.
\end{theorem}

In the case of a strictly convex control set $U$, the theorem above admits the following corollary.

\begin{corollary}\label{Cor_ScostUsc} Suppose that $A$ and $B$ are singletons and that $U$ be a compact and strictly convex body of $\R^2$ such that $0\in U$.  
If an optimal control exists, then it is the constant 
$\ub=[0,B-A\rangle\cap\partial U$.
\end{corollary}

\begin{remark}\label{rem:Cubar_geo}{To  identify the vector $\ub$ by a geometric construction, it could be useful to   observe that it is given by   
$$
\ub=C-A
$$
where $C$ is the point at maximal distance from the $A$ among those of the set 
     $$ [A,B\rangle\cap\partial(U+A).$$
    In practice, one has to 
   \begin{itemize}
       \item    
   ``center" the control set $U$  on $A$ by putting the origin of the reference system  on $A$,
   \item 
   trace the ray joining $A$ and $B$ to find the  boundary point $C$ at maximal distance from $A$, 
   \item recover $\ub$  by joining $A$ and $C$.  
   \end{itemize}
   \begin{figure}[h!] 
   
        \begin{minipage}{0.3\textwidth}


        \includegraphics[width=1.1\textwidth]{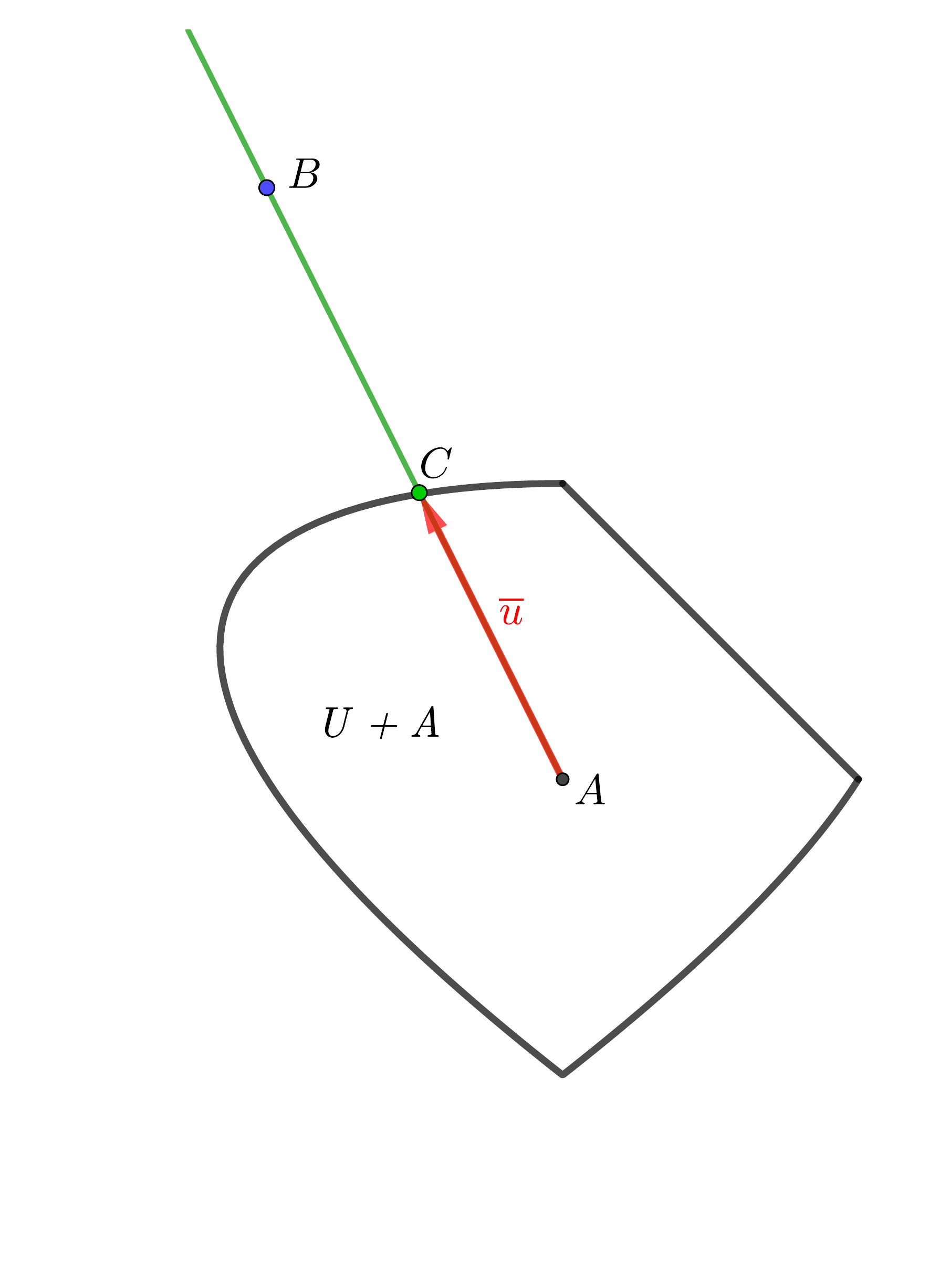}

\vspace{-9ex}
\begin{center}
      \quad\quad\quad  (a)
        \end{center}
        \end{minipage}
        \hspace{4ex}
        \begin{minipage}{0.3\textwidth} 
        \begin{center}\includegraphics[width=1.0\textwidth]{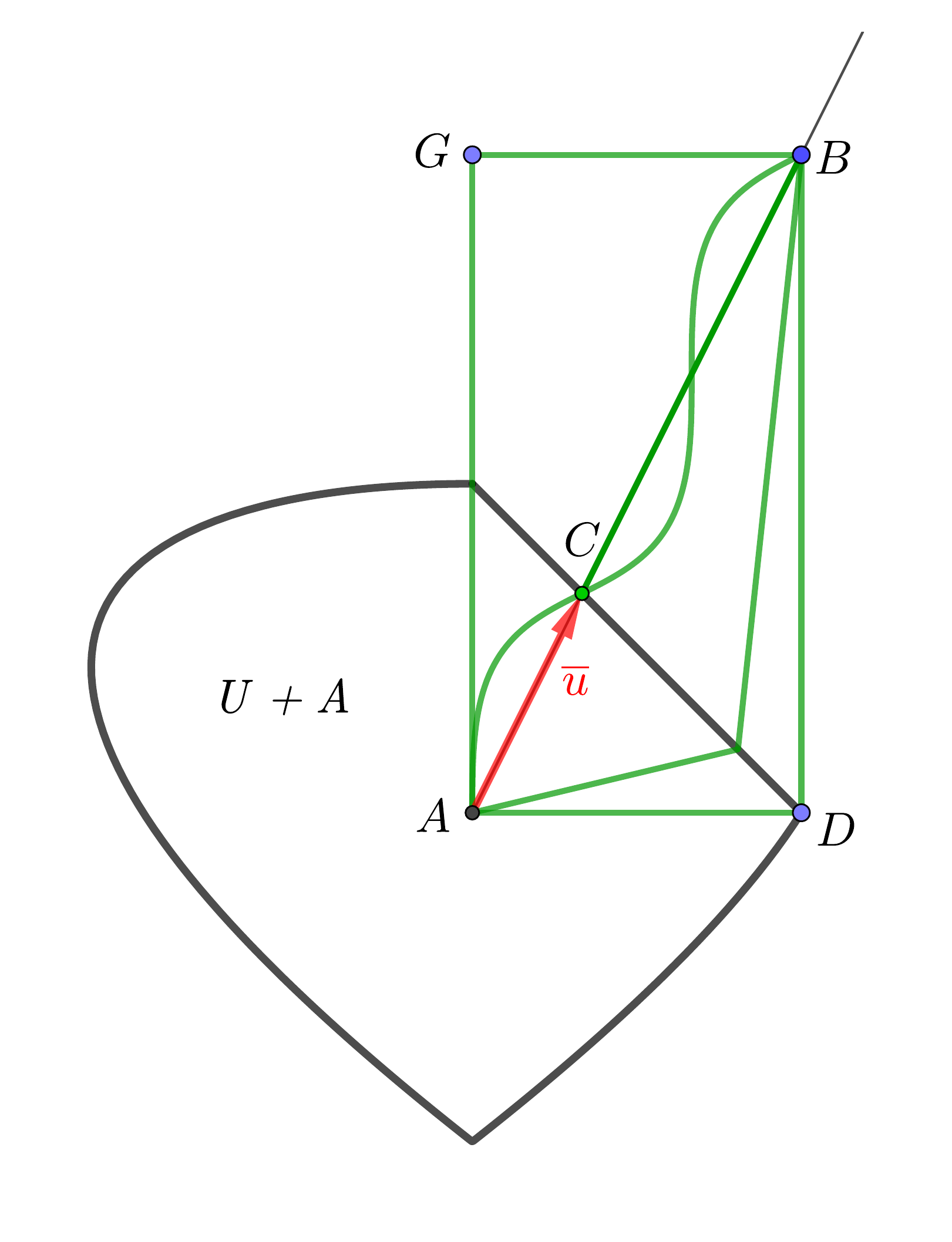}

        \vspace{-1ex}
        
        (b)
        
        \end{center}\end{minipage}\hspace{-4ex}
        \begin{minipage}{0.3\textwidth} 
        
        \begin{center}\includegraphics[width=1.2\textwidth]{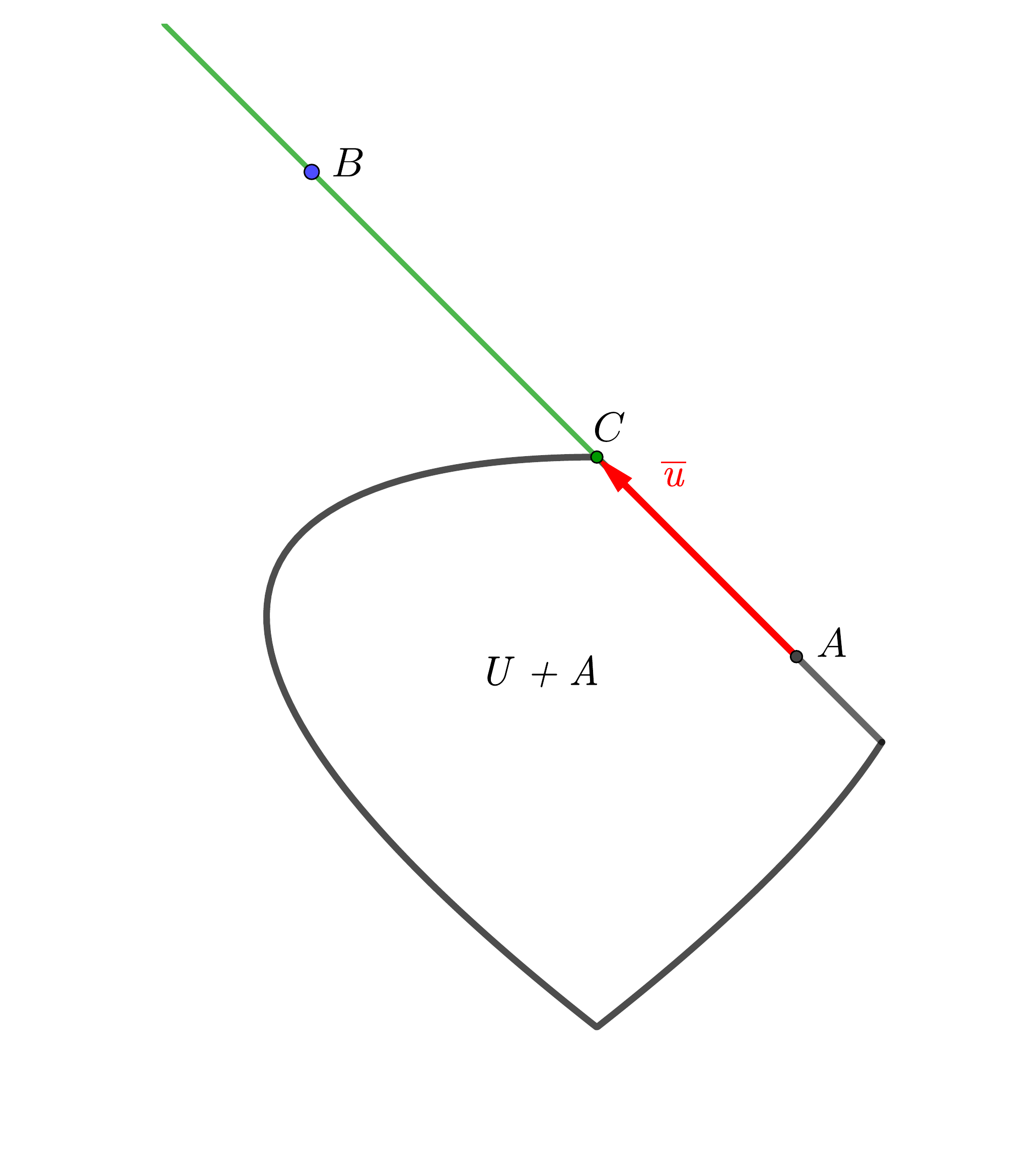}

\vspace{-4ex}

     \quad   (c)
        \end{center}
        \end{minipage}
            \caption{
             Same set $U$ (not strictly convex) but different end points $A$ and $B$,  with the corresponding constant control $\ub$ and  point $C$ (see Remark  \ref{rem:Cubar_geo}), and optimal trajectories (in green).}
        \label{fig:bordi}
    \end{figure}

   \noindent Remember that this  procedure must be applied to $U_0=U+s_0$ if  $s_0\ne0$.
   The cases shown in Figure~\ref{fig:bordi} consider the same set \(U\), but different starting and target points \(A\) and \(B\). For each case, the corresponding control \(\ub\), as given by Corollary~\ref{Cor_ScostUsc}, and the point \(C\) from the geometric construction are shown.
  In Figure~\ref{fig:bordi}(a) and~\ref{fig:bordi}(c) the constant vector $\ub$ is the unique optimal control, but in  the latter the ray joining $A$ and $B$ intersects 
   the boundary of $U$ along a segment: the point $C$ is the most distant from $A$. In Figure 
\ref{fig:bordi}(b) all measurable selections $u(t)\in \overline{S_i}$ are optimal and produce a trajectory (in red) inside the rectangle AGBD.  
To go into a numerical example corresponding to Figure 
\ref{fig:bordi}(b), let $A=(0,0)$, $B=(1,2)$, $D=(1,0)$. 
If we go from $A$ to $D$ with control $(1,0)$ and, after, from $D$ to $B$ with control $(0,1)$ then we take the same time ($t_f=3$) as going straight with control $\ub=(1/3,2/3)$. 
 Analogously, also along the trajectories $AGB$, $ADB$ and $AEB$ 
 it takes always the same time  to reach $B$. As said, this holds also for every measurable selection $u(t)\in S$ such that $x^u(0)=A$ and $x^u(t_f)=B$ as, for instance, 
for the curved trajectory in 
Figure \ref{fig:bordi}(b) 
which  corresponds to   $u(t)=\big(\frac{2}{3}\sin^2(\pi t),1-\frac{2}{3}\sin^2(\pi t)\big) $. 
}
\end{remark}

We postpone the proof of Theorem \ref{th:char_ABpoints_s=const} after  the following lemma in which   $A$ and $B$ are general ending regions (as in Section \ref{sec_Problem_setting}).

\begin{lemma}\label{lem:const_curr_Uconvex} Suppose that $s=0$ and $U$ is a compact convex body of $\R^2$, with $m$ segments in its boundary  and $0\in U$. 
    If $(u,x)$ is an optimal solution of problem $\Pcal_{A,B,U,s}$, 
    then $u$ is characterized by a global regime, regular or singular. 
\begin{enumerate}
\item\label{it:reg_s=const}    
If the regime is regular then $u$ is the (unique) constant  vector
   $$ \ub=\argmax\{|u|\ :\ u\in[0,x(t_f)-x(0)\rangle\cap\partial U\}.$$
\item\label{it:sing_cost_curr} If the regime is 
singular, then there exists $i\in\{1,...,m\}=\Ical$ such that $u(t)\in\overline{S_i}$ for a.e.\  $t\in(0,t_f)$. Moreover, if $\overline{S_i}$ is convex combination of two linearly  independent extreme points $u^i$ and $v^i$ of $U$ (satisfied if the support line of  $S_i$  does not contain the origin), 
then for  every measurable selection $v(t)\in\overline{S_i}$ we have
\begin{equation}\label{eq:tf_const_curr}
t_f=\inp{a}{x^v(t_f)-x^v(0)},
\end{equation}
where $a=\frac{(u^i-v^i)^\perp}{\langle u^{i\perp},v^i\rangle}$ and $w^\perp:=(-w_2,w_1)$.     
In particular, the final time
depends only on the initial and final points of the trajectory.
\end{enumerate}
\end{lemma}


\begin{proof} The first part of the statement is a trivial consequence of the fact that (as observed at the beginning of the subsection) the conjugate state $p$ is constant.

The Weierstrass condition \ref{eq_wc} of Lemma \ref{lem_PMP} writes 
            $u(t)\in \partial\sigma_U(p)=\argmax_{u\in U}\,\inp{p}{u}$ 
and, according to Lemma \ref{lm:geo_conv}, we have 
\begin{equation}\label{eq:weier_const_curr}
u(t)\in\begin{cases}
\{\ub\}&\mbox{if }p\not\in \cup_{i=1}^m N_i \mbox{ (regular regime)}\\
\overline{S_i}&\mbox{if }\exists\,i\in\{1,\dots,m\}\,:\, p\in N_i \mbox{ (singular regime)}
\end{cases}
\end{equation}
for a suitable constant $\ub\in U$ and for almost every $t\in(0,t_f)$. 

In particular, we have that in regular regimes the optimal control $u$ is constant. 
 Let us prove that $\ub$ takes the expression {\em(\ref{it:reg_s=const})}.  Since  $u=:\ub$ is constant and $s=0$, by the state equations 
we have 
$\ub=\frac{x(t_f)-x(0)}{t_f}$. This means that $\ub$ belongs to the half-line $[0,x(t_f)-x(0)\rangle$.  The final claimed expression  follows by the fact that (by Theorem~\ref{th:char_u_convex-case}) $\ub$ is a boundary point of $U$ 
and that, to minimize time, we have to maximize  the modulus of the  velocity by selecting the intersection point  at maximal distance from the origin. This point exists and  is unique because, by closure and convexity of $U$, the  set $[0,x(t_f)-x(0)\rangle\cap\partial U$ is a closed segment possibly reducing to a single point.

The first part of {\em(\ref{it:sing_cost_curr})} follows by Weierstrass condition (see \eqref{eq:weier_const_curr}). 
It remains only to prove the ``moreover part" of {\em(\ref{it:sing_cost_curr})}. Since $u^i$ and $v^i$ are linearly independent, then $a=\frac{(u^i-v^i)^\perp}{\langle u^{i\perp}, v^i\rangle}$ is well defined (the denominator is not $0$) and  a direct computation shows that
$\overline{S_i}\subset\{(u_1,u_2)\in\R^2\ :\ \inp{a}{u}=1\}$.     
Since $v(t)$ belongs to the segment and the latter is contained in the line $\inp{a}{u}=1$ then 
$\inp{a}{v(t)}=1$ for every $t\in[0,t_f]$. 
 On the other hand, by the state equations we have 
$v={\dot x^v}$ and, therefore,
$\inp{a}{{\dot x^v}(t)}=1$ for every $t\in[0,t_f]$
By integrating on $[0,t_f]$ we get
$\inp{a}{{x^v}(t_f)-x^v(0)}=t_f$, 
and \eqref{eq:tf_const_curr} follows. 
\end{proof}

\begin{proof}{\em of Theorem \ref{th:char_ABpoints_s=const}.}   By Lemma \ref{lem:const_curr_Uconvex} only two alternative regimes may occur. If the regime is regular the claim follows by the first assertion of the same lemma.

If the regime is singular and the support line of $S_i$ does not contain the origin, then 
the second assertion of Lemma \ref{lem:const_curr_Uconvex}  implies that   every trajectory from $A$ and $B$ generated by a measurable selection $u(t)\in \overline{S_i}$  takes the same reaching time 
$t_f=\inp{a}{B-A}$.
Moreover, the constant control
$\ub={(B-A)}/{t_f}$ 
belongs to $\overline{S_i}$. Indeed, only points $B$ such that $B-A$ falls inside the cone generated by points in $\overline{S_i}$ (i.e., $K_i:=\{\lambda u\ :\ u\in \overline{S_i},\ \lambda>0\}$) can be reached by controls that are measurable selections of  $\overline{S_i}$. Hence, $B$ must belong to $K_i$ and the half-line $[0,B-A\rangle$ intersects  $\overline{S_i}$, implying $\ub\in\overline{S_i}$.  Thus, the constant control  $\ub$
turns out to be among the optimal ones. 

It remains only the case in which the regime is singular and the support line of $S_i$  contains the origin. In such case, by the first part of assertion {\em(2)} of Lemma \ref{lem:const_curr_Uconvex} we have $u(t)\in\overline{S_i}$. On the other hand, due to the position of the origin, all such vectors $u(t)$ have the same direction (parallel to the segment joining $A$ and $B$).  The optimum is the unique among those of them that point towards $B$ and maximizes the distance from the origin, that is $\ub$.   

Then, we can claim that, independently of the regime,  the constant control $\ub$ 
 is always optimal , and that it is unique except, at most, when it belongs to $\overline{S_i}$. To completely prove the theorem it remains only to show that the uniqueness holds also when $\ub$ is an endpoint of the segment $S_i$, that is,  $\ub\in\overline{S_i}\setminus S_i=\{x_i,y_i\}$. But in such case, among all measurable selections of points of  $\overline{S_i}$, only the constant one $u(t)=\ub$ can reach $B$, and the theorem is completely proven. 
 \end{proof}

\

\subsection{The case of a linear current}

In this subsection we consider the case in which  the current $s$ is 
a linear function of the position, that is $s(x)=D^\top x$ where $D\in\R^{2\times 2}\setminus\{(0,0)\}$.

Regarding the occurrence of singular regimes, we observe that the measure-theoretic \lq\lq typical case\rq\rq~ is one in which such regimes do not occur. This is because, in view of Theorem \ref{th:sing_reg_NC}, the set of all constant $\R^{2\times 2}$ matrices possessing an eigenvector aligned with an at most countable set of directions (specifically, those normal to the boundary segments)  has a four-dimensional Lebesgue measure  zero.

Without claiming to be exhaustive, we are going to present some key examples starting from the typical cases in which singular regimes are not allowed. When needed, we assume that the 
modulus of $D^\top$ be
small enough so that there exists a neighborhood  $X$ of $A$ and $B$ in which the weak current assumption (WC) is satisfied. 

\subsubsection{$U$ strictly convex} 
Let us start by the simplest case in which $U$ is strictly convex and, hence, singular regimes (by definition) are never allowed.  

{\color{black} Suppose, moreover that  $U$ coincides locally with
the epigraph of a function of class $C^3$ with strictly positive definite
hessian. As already observed in Remarks  \ref{rem:support_nonC2} and   \ref{rem:ZNEstrconvex}, this implies that the support function $\sigma_U\in C^3(\R^2\setminus\{0\})$ and Corollary   \ref{cor:strictly_convex} applies.

Since $\nabla_x s=D^\top$ (constant),   any optimal control $u$ must be $C^2$-smooth,} $u\in\partial U$ and satisfies Zermelo equation 
    \begin{equation}
    \label{eq_ZNE1_sl} 
       \omega\big(\dot u,\ddot u-D^\top \dot u\big)=0.
    \end{equation}
Let us, first, consider the {\em isotropic case} $D^\top=\eps I$, with $\eps\in\R$ small enough. Let us note that, as for a constant current, also this case is abnormal (i.e., $p_0=0$). Indeed, in such  particular case,   equation~\eqref{eq_ZNE1_sl} takes the simpler form
$\omega(\dot u,{\ddot u})=0$. This means that the curvature of $u$ is $0$,  that is $\dot u=0$. By the strict convexity of $U$ and the fact that the optimal control belongs to the boundary, we conclude that $u$ must be constant.  It is worth noting that, by the way, the trajectory corresponding to this optimal control  will be not rectilinear  if $\eps\ne0$. 

Let us consider, now, the  {\em anisotropic case}  $d_{11}=\varepsilon$, $d_{12}=d_{21}=d_{22}=0$. 
In such case the non-parametric navigation equation
becomes (in components) 
    \begin{equation}
    \label{eq_ZNE1_d11} 
        \ddot u_1\dot u_2-\dot u_1\ddot u_2-\varepsilon \dot u_1\dot u_2=0.
    \end{equation}
By performing  a standard polar parametrization with $V=V(\theta)$ as in Example \ref{ex_ZNEestended}, 
the Zermelo navigation equation corresponding to \eqref{eq_ZNE1_d11} in polar coordinates would be (see also Remark \ref{rem:ZNEstrconvex}(d)) 
\begin{equation}
\label{eq_ZNE_aniso}
    \dot\theta=\frac{\varepsilon}{d_\kappa(\theta)}\big(V'(\theta)\sin\theta+V(\theta)\cos\theta\big)\big(V'(\theta)\cos\theta-V(\theta)\sin\theta\big).
\end{equation}
The following example shows that this is possibly not the best parametrization. 

\begin{example}[shifted ellipsis]\label{ex_shift_ellip}{Let us consider the elliptic control set $$U=\{u\in\R^2\ :\ \frac{(u_1-c_1)^2}{a_1^2}+\frac{(u_2-c_2)^2}{a_2^2}\le1\}$$ 
    centered in $c=(c_1,c_2)$, with $a_1,a_2>0$, and $\frac{c_1^2}{a_1^2}+\frac{c_2^2}{a_2^2}<1$  so that $0\in\,\inter{U}$. For simplicity, let us suppose that $A=(0,0)$, that is the starting region is put in the origin. Then, under the linear current field $s(x)=(-\eps x_1,0)^\top$, with $0<\eps<1$ the  position $x(t)$ must satisfy the Cauchy problem 
\begin{equation}
    \label{eq:CP_shifted_ell}
    \begin{cases}
    \dot x_1(t)=u_1(t)-\eps x_1(t)\\
    \dot x_2(t)=u_2(t)\\
    x_1(0)=x_2(0)=0.
    \end{cases}
    \end{equation}
     Since the current becomes stronger  as $x_1$ increases, only a subset $\Rcal$ of the plane can be reached in finite time by starting from the origin.  In fact, by estimating $u_1$ and the corresponding component  $x_1$ of the solution of the Cauchy problem,  
     one easily finds that 
\begin{equation}\label{eq:reachable_set_R}
\Rcal=\Big(\frac{1}{\eps}\Big(c_1-\frac{a_1}{a_2}\sqrt{a_2^2-c_2^2}\Big),\frac{1}{\eps}\Big(c_1+\frac{a_1}{a_2}\sqrt{a_2^2-c_2^2}\Big)\Big)\times \R.
\end{equation}
As a consequence, a solution of the minimum time problem exists if and only if $B\cap \Rcal\ne\varnothing$. For simplicity, let us suppose, from now on, that the target region be a singleton $B=(b_1,b_2)\in \Rcal$.

    The particular geometry of the domain suggests  to parametrize the boundary $\partial U$ with the angle \(\alpha(t)\) in the following way
    \[
        u_1(t)=a_1\cos\alpha(t)+c_1,\qquad u_2(t)=a_2\sin\alpha(t)+c_2.
    \]
    This {\em elliptic parametrization} is, clearly, different from the circular one of Example \ref{ex_ZNEestended}.

    By computing derivatives and substituting in {\em(a)} and 
    {\em(b)} of Corollary   \ref{cor:strictly_convex}, we obtain that 
    \begin{itemize}
     \item[(a)] for any $t\in[0,t_f]$, $\alpha'(t)=0$ or
        \begin{equation}
            \dis\alpha'(t)=-\varepsilon\sin(\alpha(t))\cos(\alpha(t)).\label{eq_ZNElce_ls_ci}
        \end{equation}
        \item[(b)] {either $\alpha'(t)\big(a_1a_2+a_2c_1\cos\alpha(t)+a_1c_2\sin\alpha(t)
            -\eps a_2  x_1(t)\cos\alpha(t)\big)\equiv0$, or it is $\ne0$ for every $t\in(0,t_f)$
     }
    \end{itemize}
    It is worth noting that \eqref{eq_ZNElce_ls_ci} looks much simpler than \eqref{eq_ZNE_aniso}: the two equations play the same role, but are obtained with different parametrizations $\alpha$ and $\theta$. In fact, the two angles have not the same physical/geometric  interpretation: $\theta$ physically represents the heading  angle and can be obtained from $\alpha\in(0,\pi/2)$ (which has only a geometrical meaning) by the relation  
    $$
    \theta=\arctan(\frac{a_2}{a_1}\tan\alpha),
    $$
and similar relations holds for the other values of $\alpha\in[0,2\pi)$.
By straightforward  estimates, the reader can easily check that,  for any point $B$ belonging to the reachable set $R$,   
{the factor inside the round brackets in} (b) never vanishes. 
 Thus, {\em the  optimal controls must be constant or satisfy \eqref{eq_ZNElce_ls_ci}.} 
    The latter is a separable differential equation and  admits the general solution 
    $$
        \alpha(t)=\arctan(C\e^{-\eps t})
    $$
    depending on the constant $C\in\R$. 
     The corresponding family of velocities $u(t)$ (written in rational form) is
     \begin{equation}
    \label{eq_opt_rat_ci_dec}
    u(t)=\frac{1}{\sqrt{1+C^2\e^{-2\eps t}}}\Big(a_1,\;{a_2C\e^{-\eps t}}\Big)+(c_1,c_2),\qquad C\in\R.
    \end{equation}
    Summarizing, we have that the optimal controls are either of this form, or constants. 

     Since the state equations are linear, the Cauchy problem \eqref{eq:CP_shifted_ell} 
    can be explicitly solved. 
    By substituting the family of controls $u=(u_1,u_2)$  computed above into such solutions, we get 
    \begin{equation}
    \label{eq:constant_control_case_ci_dec}
        (x_1(t),x_2(t))=\left(\frac{u_1}{\varepsilon }(1-e^{-\varepsilon t}) \;,\; u_2t\right), \qquad t\ge0,
    \end{equation}
    in the case of constant controls $u_1,u_2\in\R$, and
    \begin{equation}    
    \label{eq_non_constant_opt_ci_ed}
        \begin{split}
            x_1(t) &= \frac{a_1}{\varepsilon}\Big(\sqrt{1+C^{2}e^{-2\varepsilon t}}-e^{-\varepsilon t}\sqrt{1+C^{2}}\Big)+c_1(1-\e^{-\eps t}),\\
            x_2(t) 
            &= \frac{a_2}{\varepsilon}\log\frac{C+\sqrt{1+C^2}}{C e^{-\varepsilon t}+\sqrt{1+C^2e^{-2\varepsilon t}}}+c_2t
        \end{split}
    \end{equation}
    with $C\in\R$, for controls of the form \eqref{eq_opt_rat_ci_dec}.

 {The problem can be solved, for a chosen target point $B$ in the reachable set $\Rcal$ (see \eqref{eq:reachable_set_R}), by  explicitly finding the trajectories  connecting $A$ and $B$ and computing the final time. To make an explicit numerical example, suppose that $B=(1,1)$ 
 and choose  $\varepsilon=1/2$, $a_1=1$, $a_2=2$, $c_1=0$ and $c_2=-1$. This is of some interest in practical navigation, because it models a case in which a circular set of velocities is deformed into an ellipsis with center displaced from the origin due to the effect of waves as in Figure \ref{fig:waves} (see \cite[Figure 8]{sake}). Straightforward calculations, that we omit to be concise, show that the unique solution with a constant control is given by  
 \begin{equation}
    \label{eq:constcontrl_ci}
        u_1=\frac{\eps}{1-\e^{-t_f/2}},\quad u_2=\frac{1}{t_f},
    \end{equation}
 with an estimated final time 
        $t_f\approx  2.4407\pm 10^{-4}$. 
By contrast, there is a unique non-constant solution of the form \eqref{eq_non_constant_opt_ci_ed} with 
$C\approx1.8847\pm 10^{-4}$ 
and $t_f\approx2.3600\pm 10^{-4}$. 
 }

     \begin{figure}[h!] 
        \centering
        \includegraphics[trim={0pt 500pt 0pt 0pt},clip,width=0.35\textwidth]{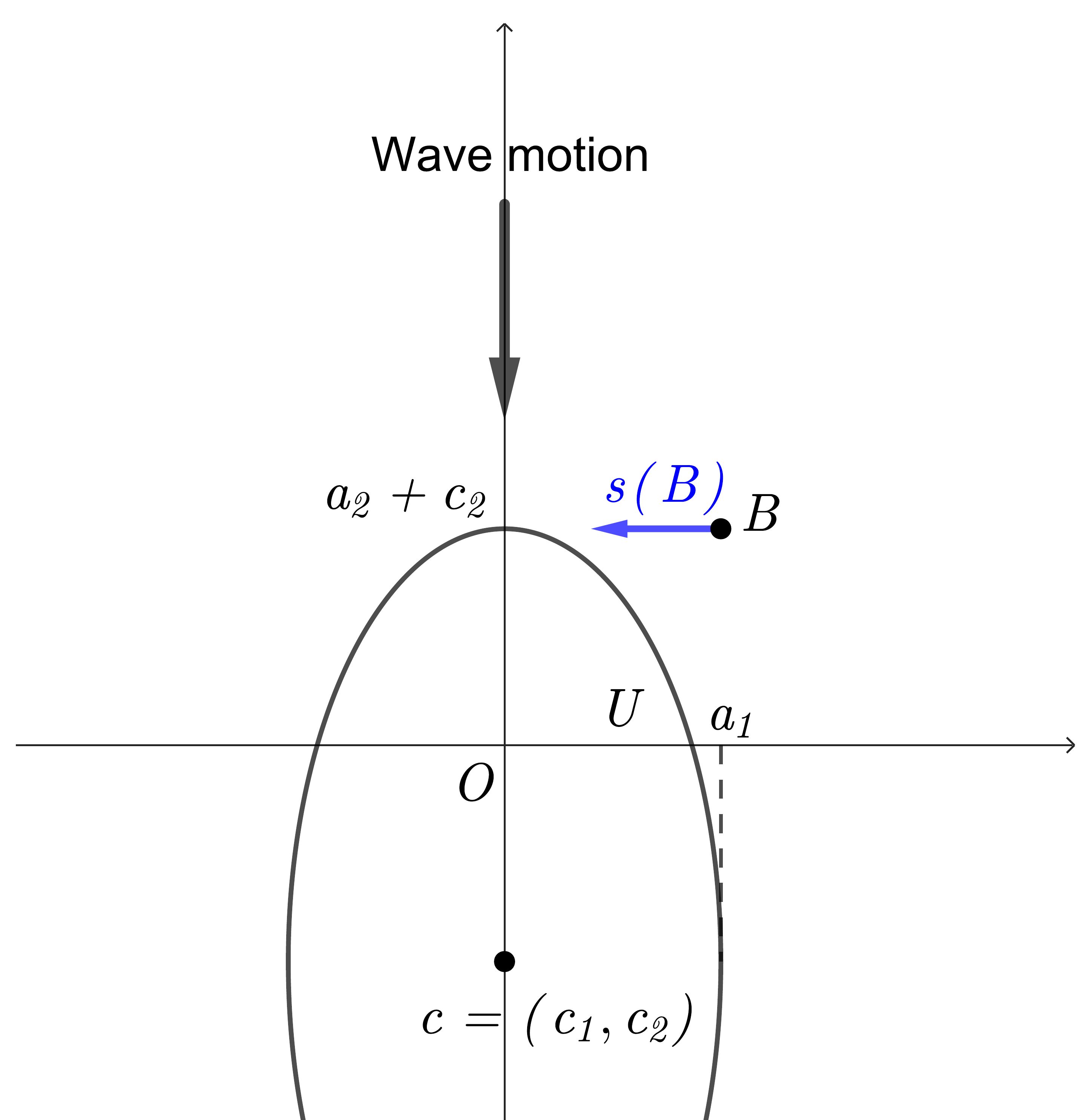}
        \vspace{2ex}
        \caption{The set of velocities $U$ deformed by the action of waves}
        \label{fig:waves}
    \end{figure}

\noindent Clearly, in the latter case, the final time is smaller compared to that required by the constant solution and, hence, the non-constant one  corresponds to  the minimum-time solution.  In this  example, the unique optimal control is given by \eqref{eq_opt_rat_ci_dec} with $C\approx1.8847\pm 10^{-4}$. 
It is worth noting that, also in the general case with $c_1=0$,  the optimal control satisfies a simple feedback control law that can be obtained by solving the first equation \eqref{eq_non_constant_opt_ci_ed}  with respect to $\e^{-\eps t}$, which gives
    \begin{equation*}
        \e^{-\eps t}=\sqrt{1+\eps^2C^2a_1^{-2}x_1^2}-\eps x_1a_1^{-1}\sqrt{1+C^2}:=h(x_1),
    \end{equation*}
and substituting into the expression of the optimal control  \eqref{eq_opt_rat_ci_dec} to get
   \begin{equation}
    \label{eq:feedbackCL}
        u(x_1)=\frac{1}{\sqrt{1+C^2h^2(x_1)}}\Big(a_1,\;a_2Ch(x_1)\Big)+(0,c_2).
    \end{equation}

    \noindent  Figure~\ref{fig:PenUltima} shows the optimal trajectory (in green) and the suboptimal one (in red)  corresponding to the constant control \eqref{eq:constcontrl_ci}, while the dashed red line is the segment  joining $A$ and $B$. In the same picture are also displayed (in blue) the vectors $s$ (current) and $u$ in different points of the trajectories.
    \begin{figure}[h!] 
        \centering
        \includegraphics[width=0.45\textwidth]{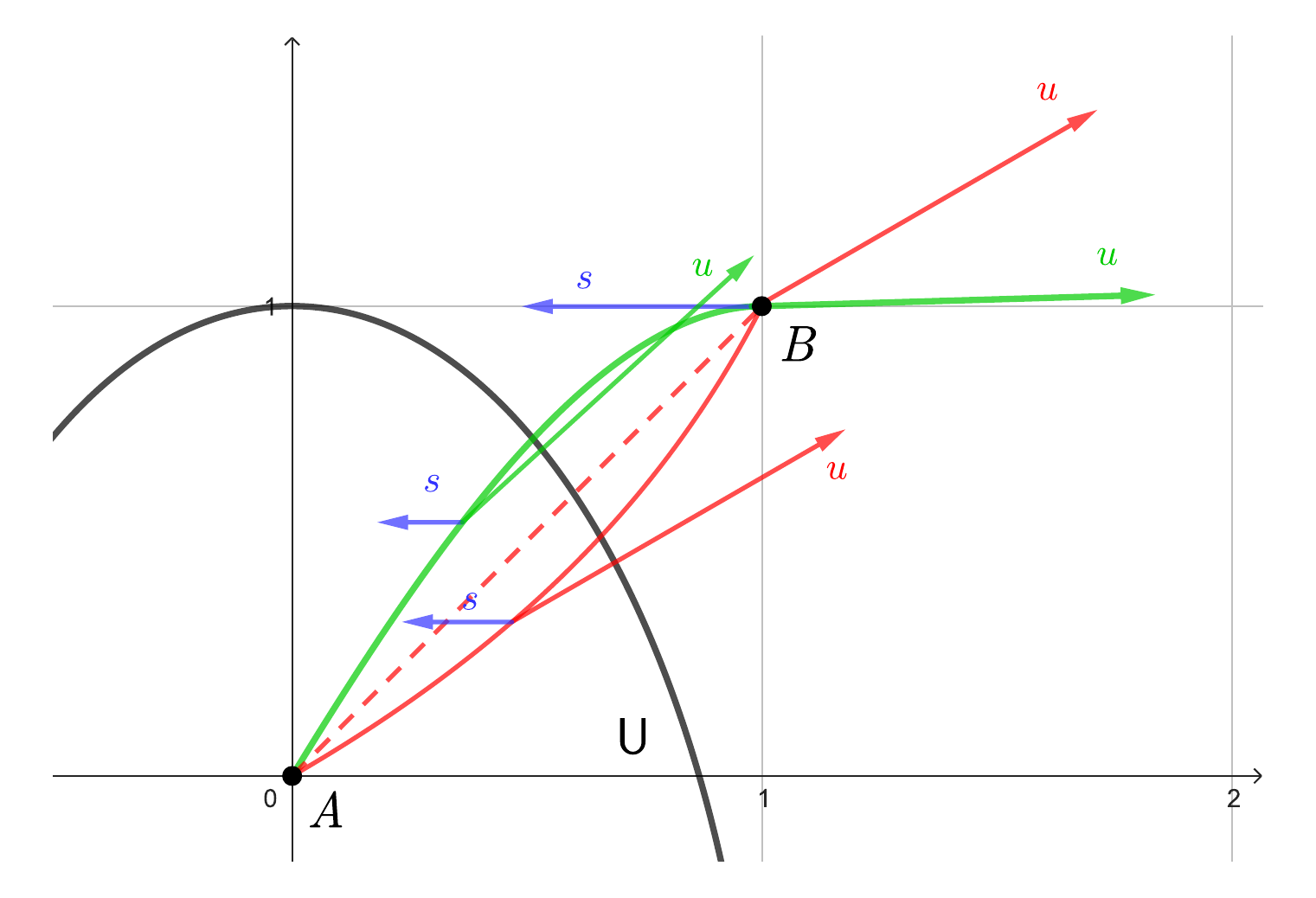}
        \caption{The optimal trajectory (in green) and the (suboptimal) one  generated by the constant control~\eqref{eq:constcontrl_ci} (in red).}
        \label{fig:PenUltima}
    \end{figure}
    Let us conclude by noting that, in practical navigation, also the suboptimal solution with $u$ constant is interesting, because it is easy to follow with a compass on conformal  charts.
}
\end{example}


{
\begin{remark}\label{rem:BCW23comparing} {It is worth noting that, Example \ref{ex_shift_ellip} falls into the Euclidean framework of \cite{BCW2023} with (using the notation therein)  {$F_0=(-\eps x_1+c_1,c_2,0)$,}   {\color{black}$F_1=a_1e_1$ and $F_2=a_2e_2$.} The determinants introduced in \cite[Proposition 3.9]{BCW2023} turn out to be
$$
D=a_1a_2,\quad {D'=\eps a_1a_2\sin\alpha\cos\alpha,}\quad D''=a_1a_2+a_2c_1\cos\alpha+a_1c_2\sin\alpha
            -\eps a_2  x_1\cos\alpha.
$$
By parameterizing the control as in Section \ref{sec_paraboundary} with  $\gamma(\alpha)=(a_1\sin\alpha,a_2\cos\alpha)$, and using the notation of that section,  we have
$D=d_\kappa$, the Zermelo equation is
{$D\dot\alpha(t)=-D'(\alpha(t))$}  and the additional (normality/abnormality) condition writes: either 
$\dot\alpha(t)D''(\alpha(t))\equiv0$ or it is always $\ne0$ in $(0,t_f)$, according to Proposition 3.9 of  \cite{BCW2023}. 
} In the setting of Example \ref{ex_shift_ellip} we have always $D''>0$  (for $A=(0,0)$ and any $B$) and, hence, the optimal trajectory lies always on a  {\em hyperbolic geodesic} in the sense of  \cite[Proposition 3.9]{BCW2023}).
\end{remark}

\subsubsection{A case study in the general convex case}

In this subsection we present an example in which the control set can be regarded as the convex hull of a set that could, in theory, represent the set of admissible velocities for a sailboat.

\begin{example}[convexified  sailboat]\label{ex:convex_sailboat_detailed}{Let us consider a linear current $s(x)= D^\top\! x$, with 
    $$
    \nabla s=D^\top:=\delta\left(\begin{matrix}
    1&0\\
    0&-1
    \end{matrix}\right),\quad \delta>0, 
    $$  
    and a
set $U=B(0,1)\cap\{(u_1,u_2)\in\R^2\ :\ -\frac{5}{4}\le u_1+u_2\le1\}$ like in Figure \ref{fig:U_vela_convex_2segmenti}.
\begin{figure}[h!] 
    \begin{center}
            \includegraphics[trim={0pt 0pt 0pt 0pt},clip,width=0.6
            \textwidth]
            {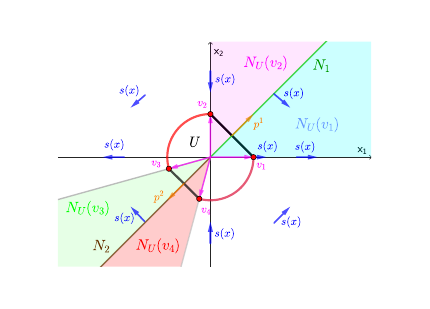}
               \end{center}
            
            \caption{The set $U$ and its normal cones. The extreme points of $U$ are displayed in red.
            }
        \label{fig:U_vela_convex_2segmenti}
    \end{figure}
 The boundary $\partial U$ contains two maximal open segments $S_1$ and $S_2$ with normal cones $N_1$ and $N_2$ generated by the vectors $p^1=(1,1)$ and $p^2=(-1,-1)$, respectively. Since they are not eigenvectors of $D$ (which are all proportional to $(0,1)$ or $(1,0)$ only),  the necessary condition of Theorem~\ref{th:sing_reg_NC} is not satisfied and {\em singular regimes cannot occur.}

    If $u(t)$ is an optimal control,    by   Theorem \ref{th:affine_regimes}, 
    the interval $[0,t_f]$ 
    can be split into the union of at most two contiguous regular regimes
    because the eigenvalues of $D$ are real. 
     According to  Weierstrass condition,  
$p(t)\in N_U(u(t))$ for a.e.\ $t\in(0,t_f)$.
 To characterize $u(t)$ by means of this cones it is enough to consider only the extreme points of $U$, since the corresponding family of normal cones covers the plane by Lemma \ref{lem:NUcovering}. 
We clearly have 
$$\ext(U)=\partial U\setminus({S_1}\cup {S_2}),$$
displayed in red in Figure~\ref{fig:U_vela_convex_2segmenti}.
This set {\color{black}is} composed by two circular arcs including the endpoints $v_1=(1,0)$, $v_2=(0,1)$, $v_3$ and $v_4$. The normal cones in these points are displayed in Figure~\ref{fig:U_vela_convex_2segmenti}, while
outside them they are simply the normal rays to the boundary. Thus, according to Weierstrass condition, we have
\begin{equation}\label{eq:uvi}
    u(t)=\begin{cases}
    v_i&\mbox{if }p(t)\in N_U(v_i),\ i\in\{1,2,3,4\},\\ 
    \frac{p(t)}{|p(t)|}&\mbox{otherwise},
    \end{cases}
\end{equation}
and $N_U(v_1)=\{(x_1,x_2)\in\R^2\ :\  |x_2|\le x_1\}$,  $N_U(v_2)=\{(x_1,x_2)\in\R^2\ :\ x_2\ge|x_1|\}$,  
$N_U(v_3)=\cone(v_3,p^2)$,  $N_U(v_4)=\cone(p^2,v_4)$,
where $\cone(\cdot,\cdot)$ denotes the {\em positive hull} generated by the two vectors, that is, the set of all positive combinations generated by them (see \cite[Section 1.1]{Schneider1993}). 
Note that equation \eqref{eq:uvi} defines $u(t)$ (almost everywhere) in an univocal way. Indeed, as Lemma \ref{lem:NUcovering} states, the normal cones corresponding to extreme points intersects only along 
the normals $N_1$ and $N_2$ to the segments on the boundary, which is exactly the region in which $p(t)$ can stay only in the (possible)  transition instant  between two contiguous regular regimes (otherwise a singular regime would occur, contradicting the fact that such regimes have been previously excluded).

Summarizing, any optimal strategy must have a structure that combines the  controls of the set $\ext(U)$  depending on the evolution of the adjoint state $p(t)$. Such evolution is driven by the adjoint system \eqref{eq:AdjointEquation}
that, in our case,  
admits the general solution
\begin{equation}
    \label{eq:gs_adj_lincur}
p_1(t)=c_1\e^{-\delta t},\quad p_2(t)=c_2\e^{\delta t}
\end{equation}
where $p(0)=(c_1,c_2)\in\R\setminus\{(0,0)\}$,  
since $p\ne0$ (see Lemma \ref{lem_PMP}). 
A relevant fact about the adjoint system for the case under consideration,
as one can immediately see by the general solution \eqref{eq:gs_adj_lincur}, 
is that 
\begin{equation*}
\mbox{{\em every quadrant is an invariant set for $p(t)$}.}
\end{equation*} 
The set in which the evolution will be actually confined  is determined by $p(0)$ or $p(t_f)$, which must be normal to the ending regions $A$ and $B$. On the other hand, if  both regions reduce to single points, no information are given and the evolution of $p(t)$  might, potentially, take place in each one of the 4 invariant sets. Once having determined in which quadrant $Q_i$ ($i\in\{1,2,3,4\}$) the evolution takes place, according to Lemma \ref{lem:invariance} whose assumption \eqref{eq:ass_UQ} is satisfied, the optimal strategy $u(t)$ can select (for a.e.\ $t$)  only the extreme controls belonging to $Q_i$, that is (see Figures    \ref{fig:U_vela_convex_2segmenti} and \ref{fig:U_vela_extQ})  
$$
\ext(U)\cap Q_i=
\begin{cases}
\{v_1,v_2\}&\mbox{if }i=1,\\
\mbox{the arc }\partial U\cap Q_2 &\mbox{if }i=2,\\
\mbox{two disconnected small arcs}&\mbox{if }i=3,\\
\mbox{the arc }\partial U\cap Q_4 &\mbox{if }i=4.
\end{cases}
$$
\begin{figure}[h!] 
        \begin{minipage}{0.48\textwidth}
    \begin{center}
            \includegraphics[width=0.9\textwidth]{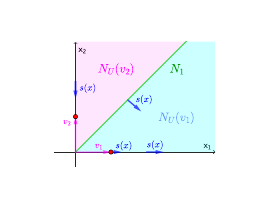}
\vspace{-4ex}

            (Q1)
            
               \end{center}
               \end{minipage}
                    \begin{minipage}{0.48\textwidth}
    \begin{center}
            \includegraphics[width=0.8\textwidth]{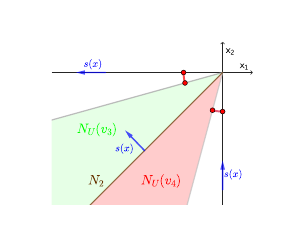}
\vspace{-2ex}

            (Q3)
            
               \end{center}
               \end{minipage} 



            

            \caption{In red, the sets $\ext(U)\cap Q_1$ and $\ext(U)\cap Q_3$}
        \label{fig:U_vela_extQ}
    \end{figure}

            




            

\noindent Let us now continue our analysis in the case in which {\em $A$ is the origin  and $B=(b_1,b_2)$ is a singleton.} 
Since the starting point $A=(0,0)$ belongs to every quadrant,  the control system under consideration satisfies also the following further invariance condition
\begin{equation}\label{eq:px_invariance}
\mbox{{\em the quadrant  containing $p(t)$ is invariant also for the trajectory $x(t)$}.}
\end{equation}
{The reader can easily verify this fact by checking that, when $x(t)$ lies on the boundary of the chosen quadrant, the corresponding vector field does not point outward; hence the trajectory either remains on the boundary or enters the interior.} 


 We could say that the quadrants  are {\em mutually invariant sets}. 
A remarkable consequence of  mutual invariance it that {\em $p(t)$ (and hence $x(t)$) must evolve in the same mutually invariant set containing the target point $B=x(t_f)$.} Indeed, if by contradiction $p(t)$ would evolve in a different invariant set  then, by \eqref{eq:px_invariance}, also the trajectory (that starts from $x(0)=0$ belonging to all invariant sets) would evolve in the same invariant set and could not reach $B$.       

As a consequence, {\em if $B$ belongs to the intersection of two invariant sets (such as the axes), also the optimal trajectory $x(t)$  must belong to the intersection.} For instance, if $B=(0,2)$ then the optimal control must be $u(t)=v_2=(0,1)$ for every $t$. This corresponds to case {\em(\ref{it:ZNE0})} of   Theorem \ref{th:char_u_convex-case}: there exists a global regular regime of abnormal kind.   

Summarizing, in general, states and co-states evolve inside the same quadrant $Q_i$ containing the target point $B$ and the optimal control is a measurable selection of controls in the set $\ext(U)\cap Q_i$ (see Figure~\ref{fig:U_vela_extQ}).  

 Besides the, already mentioned,  trivial cases in which $B$ belongs to the axes (intersection between two quadrants), we can conclude that 
\begin{itemize}
\item  {\em if $B$ belongs to the interior of the  first quadrant then the optimal controls are bang-bang combinations of the constant controls $v_1$ and $v_2$ with   at most one switching point (corresponding to a tacking maneuver in practical sailing).} Indeed, such controls are characterized to be union of at most two contiguous regular regimes in each of which $p(t)$ evolves inside one of the cones $N_U(v_1)$ or $N_U(v_2)$ (see Figure \ref{fig:U_vela_extQ}) and the control $u(t)$  takes the constant value $v_1$ or $v_2$, respectively.  The transition between the two  regimes may happen at most  once   when $p(t)$ belongs to the diagonal $N_1=N_U(v_1)\cap N_U(v_2)$.   
\item {\em If $B$ belongs to the interior of $Q_3$ then $u(t)$ is 
characterized by the union of at most two contiguous regular regimes in each of which $u$ is a smooth 
 selection from one among the  small arcs $A_1$ and $A_2$.
  Switching from an arc to the other {may occur at most once and}   corresponds to  a jibing maneuver in practical sailing. }
\item {\em If $B$ belongs to the interior of the other two quadrants then $u(t)$ is a smooth selection from the connected arc of controls belonging to that quadrant 
and the optimal control results in a global regular regime  
governed by ZNE;} no tacks or jibes occur when navigating  in these  quadrants.
\end{itemize}
When $B=(b_1,b_2)\in Q_1$, the (unique) optimal control can be explicitly computed. 
In fact, by using $v_1$ and $v_2$ we can construct two different strategies with at most one switching point, as follows: 
\begin{enumerate}
\item[(S1)] starting by going downstream with $v_1$:   
\begin{enumerate}
\item[-] compute the trajectory that starts from the origin with control $v_1$,
\item[-] compute the trajectory that a certain time $t_f$ passes through the point $B$ with control $v_2$,
\item[-] find the switching point at the intersection of the two trajectories computed above;
\end{enumerate}
\item[(S2)]  starting by going upstream with $v_2$: as before but  exchanging the controls $v_1$ and $v_2$.
\end{enumerate} 
{On the other hand, the strategy (S2) turns out to be suboptimal by observing that the triangle $T_1=\{(x_1,x_2)\ : x_2>x_1>0\}$ is invariant for $p$ (simply looking at the flow of $p(t)$, see Figure \ref{fig:U_rombo_opt}(a)).

\begin{figure}[h!] 
        \centering
             \includegraphics[width=0.45
            \textwidth]{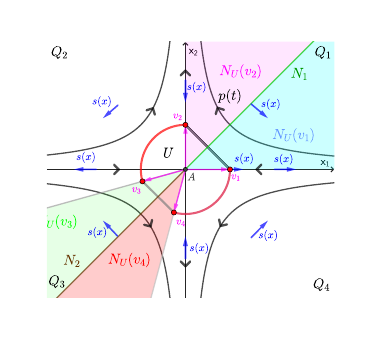}\qquad
\includegraphics[width=0.45\textwidth]{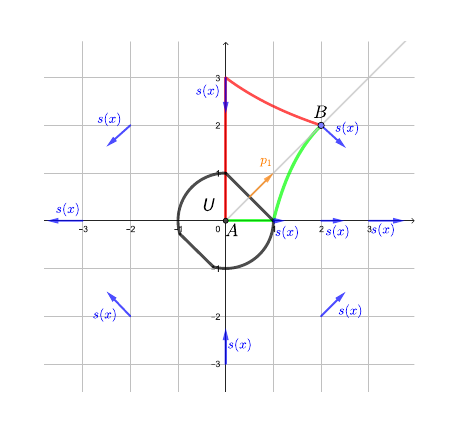}

(a)\hspace{52ex}(b)

            \caption{(a) The flow of the adjoint state $p(t)$.\ \ 
            (b) Optimal (in green) and suboptimal (in red) trajectories   for $\delta=1/4$ and $s_0=0$}
        \label{fig:U_rombo_opt}
    \end{figure}
}
            

{By computing for $\delta>0$ small enough (for brevity we omit the details) we found that
with strategy (S1) a solution exists for any $\delta<{1}/{b_2}$ and the final time turns out to be 
$$
t_f^{(1)}=\frac{1}{\delta}
\log\frac{1+b_1\delta(1-b_2\delta)}{(1-b_2\delta)}.
$$}
We can conclude that the strategy (S1) is optimal and the corresponding bang-bang  optimal control consists of two contiguous regular regimes corresponding to the open intervals $(0,t_*)$ and $(t_*,t_f)$.

For example, for $B=(2,2)$ and $\delta=1/4$ we have  $t_f^{(1)}=4\log\frac52$ and $t_f^{(2)}=4\log6$; 
the trajectory corresponding to the best strategy (S1) is displayed in green in  
 Figure \ref{fig:U_rombo_opt}(b), while that corresponding to  (S2) is displayed in red. 

 {\color{black} We conclude the analysis of this case by observing that the support function of $U$ is easily found to be
 $$
 \sigma_U(p)=\begin{cases}
 \max\{p_1,p_2\}&\mbox{ if }p\in Q_1,\\
 \sqrt{p_1^2+p_2^2}&\mbox{ if }p\in Q_2\cup Q_4\cup Q_3\setminus(N_U(v_3)\cup N_U(v_4)),\\
\langle p,v_3\rangle&\mbox{ if }p\in N_U(v_3),\\
\langle p,v_4\rangle& \mbox{ if }p\in 
N_U(v_4).
 \end{cases}
 $$
One can easily check that $\sigma_U\in C^1(\R^2\setminus(N_1\cup N_2))$ as predicted by Lemma \ref{lem:geoZNE}, and fails to be $C^2$ on the rays from the origin and passing through the corners $v_i$, $i=1,2,3,4$.
The claimed validity of ZNE in the interior of $Q_2$ and $Q_4$ follows then by the fact that $\sigma_U$ is $ C^2(\inter{Q2}\cup\inter{Q_4})$ and that $p$ evolves inside such open sets (hence Theorem \ref{th:char_u_convex-case} applies with $P=\inter{Q2}\cup\inter{Q_4}$). 

Inside $\inter{Q_3}$ the situation is more complicated. In this case, by examining the evolution of $p$, to be precise we can say that the optimal control will evolve according to ZNE with the possible exception of at most  two instants $t_3$ and $t_4$ at which $p$ enters the cone $N_U(v_3)$ and/or exits the cone $N_U(v_4)$, respectively. At these times $u$ is still continuous as stated by Theorem \ref{th:char_u_convex-case}, but the derivative may have jump discontinuities. 
 } 
 } 
\end{example}


\begin{example} \label{ex:singular_must_occur}{Let us consider a linear  current $s(x)= D^\top x$, like in the previous Example \ref{ex:convex_sailboat_detailed}, but with a
set $U=\{(u_1,u_2)\in\R^2\ :\ |u_1|\le1,\, |u_2|\le1\}$ that is a square  (see Figure \ref{fig:U_quadrato}). 
\begin{figure}[h!] 
        \centering
            \includegraphics[width=0.36\textwidth]{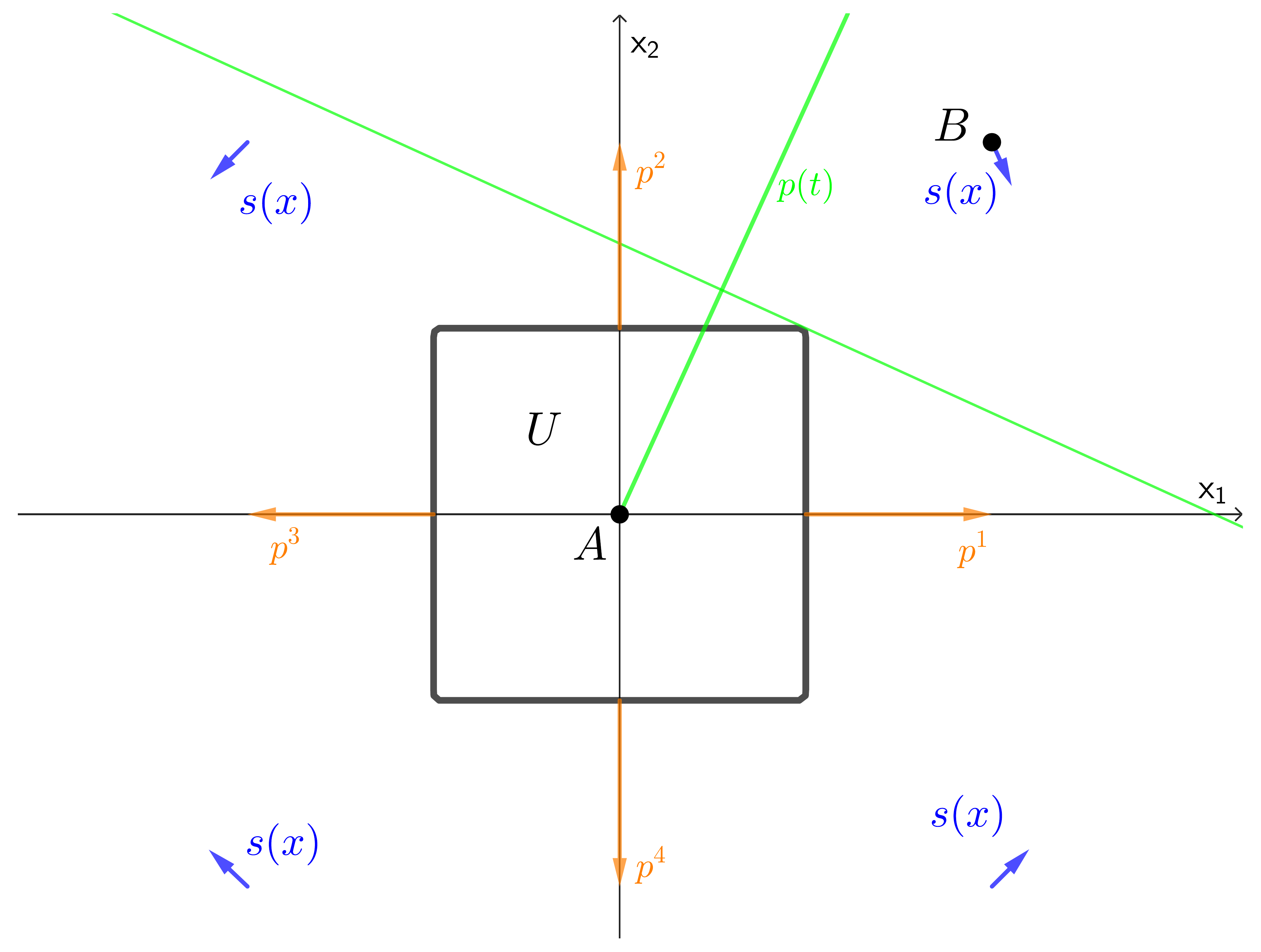}
            \caption{The case of Example \ref{ex:singular_must_occur} with $B=(2,2)$}
        \label{fig:U_quadrato}
    \end{figure}
It is a convex set with a boundary $\partial U$ containing $4$ disjoint maximal open segments $S_1$, $S_2$, $S_3$ and $S_4$ with normal cones generated by the vectors $p^1=(1,0)$, $p^2=(0,1)$, $p^3=(-1,0)$, $p^4=(0,-1)$. 
Differently from before, all such vectors are eigenvectors of $D$, then the necessary condition of Theorem  \ref{th:sing_reg_NC} is satisfied and a singular regime may occur.

Quadrants are still mutually  invariant sets.  Inside regular regimes the optimal control can select only extreme points of $U$ (see Theorem \ref{th:char_u_convex-case}).  On the other hand, in each quadrant there is  just one  extreme control (corresponding to the vertex
belonging to that quadrant) which, 
in some cases (e.g., if $B$ is along the diagonal and close enough to the origin), is not enough to reach the target. This implies the necessity of a singular regime, which must be also global by Theorem \ref{th:affine_regimes}. Actually, the optimal control must select control vectors  belonging to only one segment of the boundary.     
}
\end{example}

\section{Remarks on the non-convex case}\label{sec:nonconvex}

A comprehensive discussion of the minimum-time problem in the case of a non-convex control set $U$ would be too extensive to be addressed in this work, which is already rather substantial. Nevertheless, some remarks are in order, as they follow immediately from the results obtained in the convex case. We briefly present them in this final section.

The fact that under regular regimes the 
optimal control be a measurable selection of extreme points of $U$, that is (see Theorem \ref{th:char_u_convex-case})
$u(t)\in\ext(U)\ \mbox{ for a.e.\ }t\in[0,t_f]$, 
has the remarkable consequence that whenever singular regimes are excluded (as, e.g., in Example \ref{ex:convex_sailboat_detailed}) 
then any other minimum-time problem with the same constraints, but posed on a control set $V\subseteq U$ (possibly  not convex) containing all extreme points of $U$, has the same optimal solutions. Precisely,  the following theorem holds. 

\begin{theorem}
Let $(x,u)$ be a solution to  problem $\Pcal_{U}$ (see Section \ref{sec_Problem_setting}) with $U$ a compact convex body  of $\R^2$, {$s\in C^1$,}    and suppose that singular regimes do not occur. Assume, moreover, that  $V\subseteq\R^2$ be such that $U$ is the convex envelope of $V$.   Then $(x,u)$ is a solution also for $\Pcal_{V}$. 
\end{theorem}


\begin{proof}  Since $V\subseteq U$, then any solution $u$ to $\Pcal_{U}$ such that $u(t)\in V$ for a.e.\ $t\in[0,t_f]$ is a solution also for $\Pcal_{V}$. And this is exactly what happens because, by assumption, 
 singular regimes are excluded and therefore, by  Theorem \ref{th:char_u_convex-case}, the optimal  control $u$  is a measurable selection of extreme points of $U$, that is,
$u(t)\in\ext(U)$  for a.e.\ $t\in[0,t_f]$. 
On the other hand we have that  $\ext(U)=\ext(\co(V))\subseteq V$ (see \cite[Section A.5, Exercise 11]{hiriart-urruty_fundamentals_2004}) and, hence,  $u(t)\in V$ for almost every $t$ as claimed. 
\end{proof}

In the case of a constant current, considered in Subsection \ref{subsec:constant_current}, our Theorem \ref{th:char_ABpoints_s=const}
admits the following corollary that, for simplicity, is stated for $s=0$ (as usual, the corresponding result for a constant $s=s_0$ is obtained by replacing $U$ with $U_0$). See Figure \ref{fig:2nonconvex}(a) for an illustration of part {\em (2)} below. 

\begin{corollary}\label{cor:nonconvex_constant}
Suppose that $V$ is a compact subset of $\R^2$ and that $U=\co(V)$ is a convex body  with a finite number of segments on the boundary and such that $0\in\inter{U}$.  Suppose moreover that $A$ and $B$ are singletons and $s=0$.   
Denoting by $C$ the (unique) intersection point between $\partial (U+A)$  and the ray originating in $A$ and passing through $B$, we have:
\begin{enumerate}
    \item if $C$ in an extreme point of $U$, then $\ub=C-A$ is an optimal control and the segment $[A,B]$ is an optimal trajectory for $\Pcal_{V}$,
\item if $C$ is an interior point of a segment $S+A$ on the boundary of $U+A$, then all  trajectories  that go from $A$ to $B$ by selecting controls in $\overline{S}\cap V$  are optimal for $\Pcal_V$,  and at least one of them exists; if $\#(\overline{S}\cap V)<+\infty$ then the optimal control is piecewise constant and the optimal trajectory is piecewise linear.
\end{enumerate}
\end{corollary}


\begin{proof}
First of all, let us remark that by Corollary \ref{cor_weak_stream_ex} a solution to problem $\Pcal_{A,B,U,0}$ exists. Second, by the assumption $0\in\inter{U}$ and the convexity of $U$, the ray $[A,B\rangle$ intersects $U+A$ in exactly one point $C$. 

In case {\em (1)}, namely, if $C\in \ext(U)$, then, by Theorem \ref{th:char_ABpoints_s=const} (and the subsequent Remark~\ref{rem:Cubar_geo}), the optimal control for $\Pcal_U$ is the constant $C-A=:\ub$ and the trajectory is the segment connecting $A$ and $B$. Since  $\ext(U)=\ext(\co(V))\subseteq V$ (see \cite[Section A.5, Exercise 11]{hiriart-urruty_fundamentals_2004}), then 
$\ub\in V$ and the constant control $\ub$ is optimal also for $V$ (because $V\subseteq U$).  

In case {\em (2)}, by the second part of Theorem \ref{th:char_ABpoints_s=const} all trajectories that go from $A$ to $B$ by selecting controls in $\overline{S}$  are optimal for $\Pcal_U$. Since $V\subseteq U$ it turns out that, among all these trajectories, those that select only controls in
$\overline{S}\cap V$  are optimal for $\Pcal_V$. On the other hand, at least one of such trajectories exists because $\overline{S}\cap V$ contains at least the endpoints of the segment (which belongs to $\ext(U)$ and hence to $V$, as noted also before) and, via Bang-bang Theorem,  a trajectory ending in $B$ can be constructed by using only these two points (see \cite[Corollary 3.10.2]{bressan_introduction_2007}).  
\end{proof}

\begin{figure}[h!] 
  \hspace{4ex}      \begin{minipage}{0.40\textwidth}
            \begin{center}
 \includegraphics[width=0.8\textwidth]{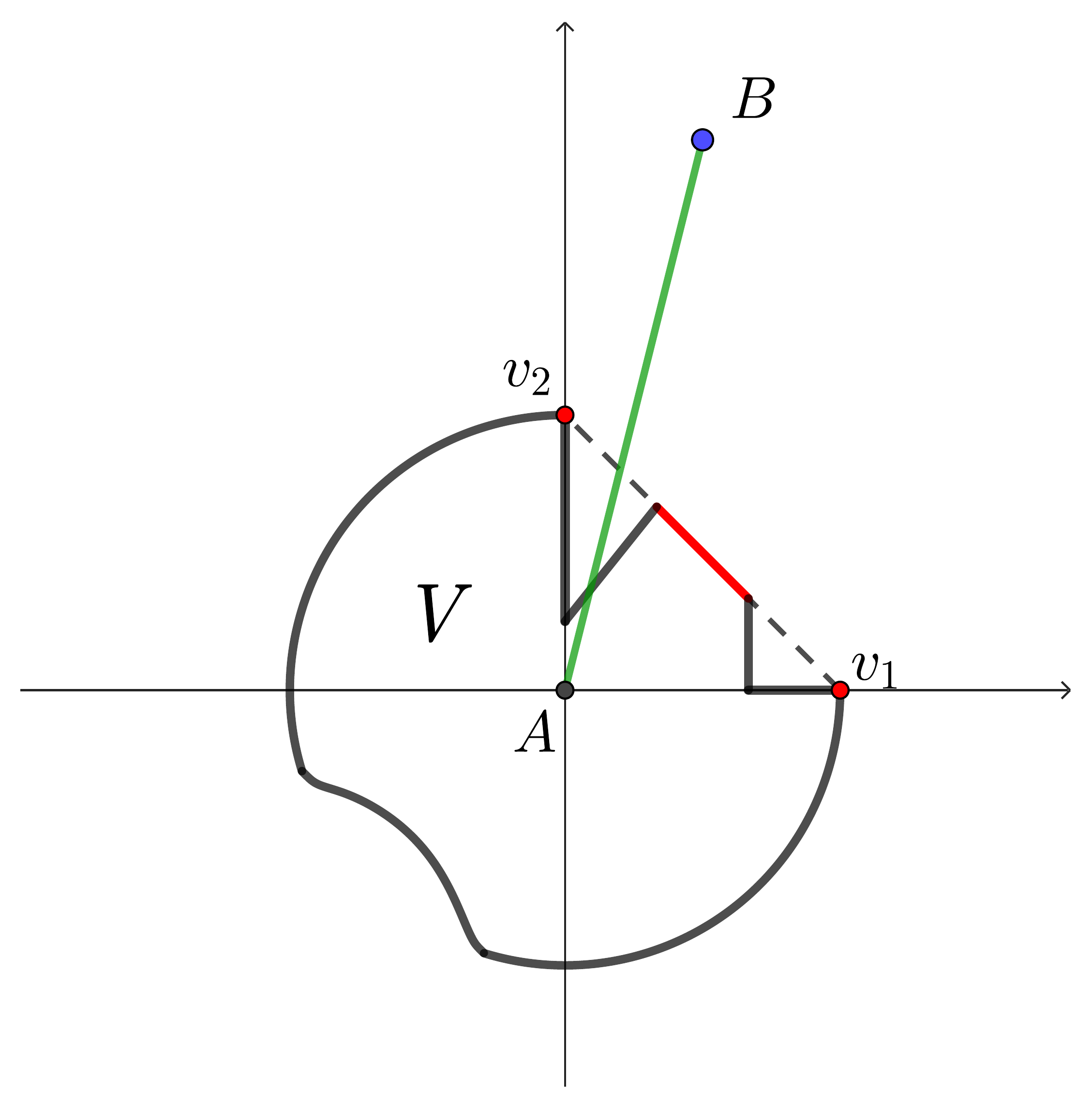}

\vspace{1ex}

(a)
\end{center}
        \end{minipage}
        \begin{minipage}{0.55\textwidth}

            \vspace{-5ex}
            \begin{center}

         \includegraphics[width=0.9\textwidth]{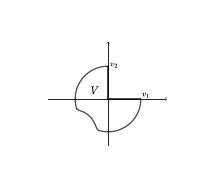}

\vspace{-5ex}

            (b)

            \end{center}
        \end{minipage}
        \caption{(a) Case {\em(2)} of Corollary \ref{cor:nonconvex_constant}.   The set $\overline{S}\cup V$ is displayed in red. (b) Typical set of velocities of a sailing boat.}
        \label{fig:2nonconvex}
    \end{figure}

            


\begin{example}\label{ex:non_multiconvex}{The case of a non-convex control set is particularly interesting in sailing navigation where the set $V$ of admissible velocities of the sailing boat is always non-convex, like in Figure~\ref{fig:2nonconvex}(b)). 
As in the picture, the convex hull typically presents two segments on the boundary.
The case of a linear current has been discussed  in Example \ref{ex:convex_sailboat_detailed} (see also Figure \ref{fig:U_vela_convex_2segmenti}). Since, there, singular regimes are not allowed, then 
 the optimal solutions of the convexified problem are optimal also for the 
non-convex navigation problem posed on $V$.
 Differently, in the case of a constant current, Corollary \ref{cor:nonconvex_constant} applies and every point $B$ of the first quadrant can be reached in minimum time by any zig-zag trajectory generated by a bang-bang combination of the extreme controls $v_1$ and $v_2$.
}
\end{example}


\begin{remark} (a) {Corollary \ref{cor:nonconvex_constant} recovers the results on multi-convex sets of Lemma 4.9 and Theorem 4.9 of \cite{MPRR2025}. On the other hand,  our result is more general because it is not restricted to multi-convex sets (see, e.g., Example \ref{ex:non_multiconvex}).  



(b) We conclude by observing that  {\em tack points}, in the sense of discontinuity points of the derivative of the optimal trajectory, can arise not only in the non-convex case (as, e.g., in the navigation of a sailing boat) but also in the (non-strictly) convex case as separation points of contiguous regimes (see Definition \ref{def:regimes}, Theorem \ref{thm:splitting}, Theorem \ref{th:char_u_convex-case} and Example \ref{ex:convex_sailboat_detailed}).
}    
\end{remark}

\section{Conclusions and future  research}

We have studied a generalized version of Zermelo’s navigation problem in which the control set  of admissible velocities is a general convex body. Under the natural assumption of weak currents, we established existence results and derived necessary optimality conditions through Pontryagin's Maximum Principle and tools from convex analysis. We proved that the time horizon can be split in regular regimes, in which 
the optimal control is differentiable and satisfies a non-parametric Zermelo-like navigation equation,  and singular regimes in which the optimal control is largely undetermined. Necessary conditions for the occurrence of such  singular regimes are investigated, and  applications are given. 
A detailed study has been devoted to the case of a position-affine current with a $t$-independent jacobian. 
We have shown that regular and singular regimes cannot coexist: the system admits either a single global singular regime or contiguous regular regimes, which reduce to at most two when the jacobian has real eigenvalues, excluding vortical flows.

The theoretical results were illustrated with numerical examples relevant in practical navigation. Some results on the non-convex case are derived as direct consequences of those obtained in the convex setting.

Future research will more deeply investigate the important case of non-convex control sets, where the lack of convexity introduces additional mathematical challenges and potentially richer control structures. Another promising direction is to extend the results presented here specifically in the planar case to higher dimensions, investigating whether analogous simplifications of the optimality conditions hold beyond the two-dimensional case.

\

\noindent
{\bf Acknowledgements.}  {\color{black}The authors are grateful to Prof.\ Dan Goreac for a helpful discussion on the subject.} 
The authors are members of GNAMPA--INdAM. 
 The present work was supported by the European Social Fund Plus (ESF+) – Regional Programme 2021-2027 of the Autonomous Region Friuli Venezia Giulia [Specific Program 22/23, Linea A - PhD Scholarship]. CUP: G23C23001130008

\bibliographystyle{abbrv}

\bibliography{freddi}

@article {ALS2021,
    AUTHOR = {Ardentov, A. A. and Lokutsievskiy, L. V. and Sachkov, Yu. L.},
     TITLE = {Extremals for a series of sub-{F}insler problems with
              2-dimensional control via convex trigonometry},
   JOURNAL = {ESAIM Control Optim. Calc. Var.},
  FJOURNAL = {ESAIM. Control, Optimisation and Calculus of Variations},
    VOLUME = {27},
      YEAR = {2021},
     PAGES = {Paper No. 32, 52},
      ISSN = {1292-8119,1262-3377},
   MRCLASS = {49K15 (53B40 53C17)},
  MRNUMBER = {4251315},
MRREVIEWER = {Ivan\ Beschastnyi},
       DOI = {10.1051/cocv/2021024},
       URL = {https://doi.org/10.1051/cocv/2021024},
}

@book {CSbook,
    AUTHOR = {Cannarsa, Piermarco and Sinestrari, Carlo},
     TITLE = {Semiconcave functions, {H}amilton-{J}acobi equations, and
              optimal control},
    SERIES = {Progress in Nonlinear Differential Equations and their
              Applications},
    VOLUME = {58},
 PUBLISHER = {Birkh\"auser Boston, Inc., Boston, MA},
      YEAR = {2004},
     PAGES = {xiv+304},
      ISBN = {0-8176-4084-3},
   MRCLASS = {49-02 (35F20 49K20 49L20)},
  MRNUMBER = {2041617},
MRREVIEWER = {Pierre\ Cardaliaguet},
}

@article {CQSP97,
    AUTHOR = {Cardaliaguet, P. and Quincampoix, M. and Saint-Pierre, P.},
     TITLE = {Optimal times for constrained nonlinear control problems
              without local controllability},
   JOURNAL = {Appl. Math. Optim.},
  FJOURNAL = {Applied Mathematics and Optimization},
    VOLUME = {36},
      YEAR = {1997},
    NUMBER = {1},
     PAGES = {21--42},
      ISSN = {0095-4616,1432-0606},
   MRCLASS = {49J15 (49K40 93C15)},
  MRNUMBER = {1446790},
MRREVIEWER = {E.\ N.\ Barron},
       DOI = {10.1007/s002459900053},
       URL = {https://doi.org/10.1007/s002459900053},
}

@article {KS1989,
    AUTHOR = {Krener, Arthur J. and Sch\"attler, Heinz},
     TITLE = {The structure of small-time reachable sets in low dimensions},
   JOURNAL = {SIAM J. Control Optim.},
  FJOURNAL = {SIAM Journal on Control and Optimization},
    VOLUME = {27},
      YEAR = {1989},
    NUMBER = {1},
     PAGES = {120--147},
      ISSN = {0363-0129},
   MRCLASS = {49E15 (49B10 93B27 93B50)},
  MRNUMBER = {980227},
       DOI = {10.1137/0327008},
       URL = {https://doi.org/10.1137/0327008},
}

@article {Sussmann1987_1,
    AUTHOR = {Sussmann, H. J.},
     TITLE = {The structure of time-optimal trajectories for single-input
              systems in the plane: the {$C^\infty$} nonsingular case},
   JOURNAL = {SIAM J. Control Optim.},
  FJOURNAL = {SIAM Journal on Control and Optimization},
    VOLUME = {25},
      YEAR = {1987},
    NUMBER = {2},
     PAGES = {433--465},
      ISSN = {0363-0129},
   MRCLASS = {49B36 (49C20 58A35 93C10 93C15)},
  MRNUMBER = {877071},
MRREVIEWER = {P.\ Brunovsk\'y},
       DOI = {10.1137/0325025},
       URL = {https://doi.org/10.1137/0325025},
}

@book{Schneider1993,
    AUTHOR = {Schneider, Rolf},
     TITLE = {Convex Bodies: the {B}runn-{M}inkowski Theory},
    SERIES = {Encyclopedia of Mathematics and its Applications},
    VOLUME = {44},
 PUBLISHER = {Cambridge University Press, Cambridge},
      YEAR = {1993},
     PAGES = {xiv+490},
      ISBN = {0-521-35220-7},
   MRCLASS = {52A39 (52-02 52A20)},
  MRNUMBER = {1216521},
MRREVIEWER = {W.\ J.\ Firey},
       DOI = {10.1017/CBO9780511526282},
       URL = {https://doi.org/10.1017/CBO9780511526282},
}

@article{RVHB2021,
    author = {Reche-Vilanova, Martina  and Hansen, Heikki  and Bingham, Harry Bradford},
    title = {Performance Prediction Program for Wind-Assisted Cargo Ships},
    journal = {Journal of Sailing Technology},
    volume = {6},
    number = {01},
    pages = {91-117},
    year = {2021},
    month = {06},
    issn = {2475-370X},
    doi = {10.5957/jst/2021.6.1.91},
    url = {https://doi.org/10.5957/jst/2021.6.1.91},
    eprint = {https://onepetro.org/JST/article-pdf/6/01/91/2478471/sname-jst-2021-05.pdf},
}

@article {FB2022,
author = {Fathi, Zohreh and Bidabad, Behroz},
title = {On the geometry of {Z}ermelo’s optimal control trajectories},
fjournal = {AUT Journal of Mathematics and Computing},
journal = {AUT J.\ Math.\ Comput.},
volume = {3},
number = {1},
pages = {1-10},
year  = {2022},
publisher = {Amirkabir University of Technology},
issn = {2783-2449}, 
eissn = {2783-2287}, 
doi = {10.22060/ajmc.2021.20459.1066},	
url = {https://ajmc.aut.ac.ir/article_4542.html},
eprint = {https://ajmc.aut.ac.ir/article_4542_b4195566cde159aa52619faf44a7255e.pdf}
}

@article {FB2024,
    AUTHOR = {Fathi, Zohreh and Bidabad, Behroz},
     TITLE = {Time-optimal solutions of {Z}ermelo's navigation problem with
              moving obstacles},
   JOURNAL = {Differential Geom. Appl.},
  FJOURNAL = {Differential Geometry and its Applications},
    VOLUME = {97},
      YEAR = {2024},
     PAGES = {Paper No. 102177, 17},
      ISSN = {0926-2245,1872-6984},
   MRCLASS = {53C60 (49K15 53B40 53C22 93C15)},
  MRNUMBER = {4795638},
MRREVIEWER = {Libing\ Huang},
       DOI = {10.1016/j.difgeo.2024.102177},
       URL = {https://doi.org/10.1016/j.difgeo.2024.102177},
}

@article {CFS2000,
    AUTHOR = {Cannarsa, Piermarco and Frankowska, H\'el\`ene and Sinestrari,
              Carlo},
     TITLE = {Optimality conditions and synthesis for the minimum time
              problem},
   JOURNAL = {Set-Valued Anal.},
  FJOURNAL = {Set-Valued Analysis. An International Journal Devoted to the
              Theory of Multifunctions and its Applications},
    VOLUME = {8},
      YEAR = {2000},
    NUMBER = {1-2},
     PAGES = {127--148},
      ISSN = {0927-6947,1572-932X},
   MRCLASS = {49K15 (49N35)},
  MRNUMBER = {1780579},
MRREVIEWER = {Vladimir\ Veliov},
       DOI = {10.1023/A:1008726610555},
       URL = {https://doi.org/10.1023/A:1008726610555},
}

@book{RockWets09,
 title = {Variational Analysis},
 publisher = {Springer-Verlag},
 year = {1998},
author = {Ralph Tyrrell Rockafellar and Roger J-B Wets},
 volume = {317},
 series = {Gundlehren der mathematischen Wissenchaften},
 address = {Berlin},
 edition = {3rd printing, 2009}
}

@book {ZermeloII,
    AUTHOR = {Zermelo, Ernst},
     TITLE = {Ernst {Z}ermelo---collected works. {V}ol. {II}. {C}alculus of
              variations, applied mathematics, and physics}, 
    VOLUME = {23},
      NOTE = {Ed.\ H.-D.\ Ebbinghaus and A.\ Kanamori},
 PUBLISHER = {Springer-Verlag, Berlin},
      YEAR = {2013},
     PAGES = {xx+781},
      ISBN = {978-3-540-70855-1; 978-3-540-70856-8},
   MRCLASS = {01A75 (01-06 03-03)},
  MRNUMBER = {3137671},
MRREVIEWER = {J. M. Plotkin},
       DOI = {10.1007/978-3-540-70856-8},
       URL = {https://doi.org/10.1007/978-3-540-70856-8},
}

@book{sake,
  title={On minimal-time ship routing},
  author={Bijlsma, Sake Johannes},
  year={1975},
  type={Phd Thesis},
   publisher={Technische Universiteit Delft, Phd Thesis}
}

@misc{MPRR2025,
      title={Time-dependent {Zermelo} navigation with tacking}, 
      author={Steen Markvorsen and Enrique Pendás-Recondo and Frederik Möbius Rygaard},
      year={2025},
      eprint={2508.07274},
      archivePrefix={arXiv},
      primaryClass={math.DG},
      url={https://arxiv.org/abs/2508.07274}, 
      note={arxiv.org/abs/2508.07274},
}

@Book{R,
author="Rudin, W.",
Title= "Analisi Reale e Complessa",
Publisher= "Boringhieri" 
}

@article{K2019,
 author = {Kopacz, Piotr},
 title = {On generalization of {Zermelo} navigation problem on {Riemannian} manifolds},
 fjournal = {International Journal of Geometric Methods in Modern Physics},
 journal = {Int. J. Geom. Methods Mod. Phys.},
 issn = {0219-8878},
 volume = {16},
 number = {4},
 pages = {19},
 note = {Id/No 1950058},
 year = {2019},
 language = {English},
 doi = {10.1142/S0219887819500580},
 zbMATH = {7100854},
 Zbl = {1425.53020}
}

@Book{B,
    author= "Buttazzo, G.",
    title=  "Semicontinuity, Relaxation and Integral Representation in the Calculus of Variations",
    series= "Pitman Res.\ Notes Math.\ Ser.",
        volume = "207",
    publisher=  "Longman",
    address=    "Harlow",
    year=   "1989"
}

@Book{P,
    author = "Pironneau, O.",
    title = "Optimal shape design for elliptic systems",
    publisher = "Springer-Verlag",
    address = "Berlin",
    year = "1984"
}

@article{BBBCG2019,
    author = {Biferale, Luca and Bonaccorso, Fabio and Buzzicotti, Michele and Clark Di Leoni, Patricio and Gustavsson, Kristian},
    title = {Zermelo’s problem: Optimal point-to-point navigation in 2D turbulent flows using reinforcement learning},
     JOURNAL = {Chaos},
  FJOURNAL = {Chaos. An Interdisciplinary Journal of Nonlinear Science},
    volume = {29},
    number = {10},
    pages = {103138},
    year = {2019},
    month = {10},
    issn = {1054-1500},
    doi = {10.1063/1.5120370},
    url = {https://doi.org/10.1063/1.5120370},
    eprint = {https://pubs.aip.org/aip/cha/article-pdf/doi/10.1063/1.5120370/14625499/103138_1_online.pdf},
}

@preamble{
   "\def\cprime{$'$} "}

@book {Clarke2013,
    AUTHOR = {Clarke, Francis},
     TITLE = {Functional analysis, calculus of variations and optimal
              control},
    SERIES = {Graduate Texts in Mathematics},
    VOLUME = {264},
 PUBLISHER = {Springer, London},
      YEAR = {2013},
     PAGES = {xiv+591},
      ISBN = {978-1-4471-4819-7; 978-1-4471-4820-3},
   MRCLASS = {49-02 (46-01 49J52 49K05 49K15)},
  MRNUMBER = {3026831},
MRREVIEWER = {Richard B. Vinter},
       DOI = {10.1007/978-1-4471-4820-3},
       URL = {https://doi.org/10.1007/978-1-4471-4820-3},
}

@book {BC2003,
    AUTHOR = {Bonnard, Bernard and Chyba, Monique},
     TITLE = {Singular trajectories and their role in control theory},
    SERIES = {Math\'{e}matiques \& Applications (Berlin) [Mathematics \&
              Applications]},
    VOLUME = {40},
 PUBLISHER = {Springer-Verlag, Berlin},
      YEAR = {2003},
     PAGES = {xvi+357},
      ISBN = {3-540-00838-1},
   MRCLASS = {93-02 (37N35 49K15 93C10 93C15)},
  MRNUMBER = {1996448},
MRREVIEWER = {Kevin A. Grasse},
}

@article {Mania1937,
    AUTHOR = {Mani\`a, Basilio},
     TITLE = {Sopra un problema di navigazione di {Z}ermelo},
   JOURNAL = {Math. Ann.},
  FJOURNAL = {Mathematische Annalen},
    VOLUME = {113},
      YEAR = {1937},
    NUMBER = {1},
     PAGES = {584--599},
      ISSN = {0025-5831},
   MRCLASS = {DML},
  MRNUMBER = {1513108},
       DOI = {10.1007/BF01571651},
       URL = {https://doi.org/10.1007/BF01571651},
}

@article {DO2015,
    AUTHOR = {Dmitruk, Andrei V.\ and Osmolovskii, Nikolai P.\ },
     TITLE = {Necessary conditions for a weak minimum in optimal control
              problems with integral equations on a variable time interval},
   JOURNAL = {Discrete Contin. Dyn. Syst.},
  FJOURNAL = {Discrete and Continuous Dynamical Systems. Series A},
    VOLUME = {35},
      YEAR = {2015},
    NUMBER = {9},
     PAGES = {4323--4343},
      ISSN = {1078-0947,1553-5231},
   MRCLASS = {49K21 (45D05 49K15)},
  MRNUMBER = {3392628},
MRREVIEWER = {M.\ Margarida Amorim Ferreira},
       DOI = {10.3934/dcds.2015.35.4323},
       URL = {https://doi.org/10.3934/dcds.2015.35.4323},
}

@article {B2019,
    AUTHOR = {Baboolal, Dharmanand and Pillay, Paranjothi},
     TITLE = {Some remarks on rectifiably connected metric spaces},
   JOURNAL = {Topology Appl.},
  FJOURNAL = {Topology and its Applications},
    VOLUME = {251},
      YEAR = {2019},
     PAGES = {107--124},
      ISSN = {0166-8641,1879-3207},
   MRCLASS = {54D35},
  MRNUMBER = {3876237},
MRREVIEWER = {Yasunao\ Hattori},
       DOI = {10.1016/j.topol.2018.10.012},
       URL = {https://doi.org/10.1016/j.topol.2018.10.012},
}

@article {CFM06,
    AUTHOR = {Cellina, Arrigo and Ferriero, Alessandro and Marchini, Elsa Maria},
     TITLE = {On the existence of solutions to a class of minimum time
              control problems and applications to {F}ermat's principle and
              to the brachystocrone},
   JOURNAL = {Systems Control Lett.},
  FJOURNAL = {Systems \& Control Letters},
    VOLUME = {55},
      YEAR = {2006},
    NUMBER = {2},
     PAGES = {119--123},
      ISSN = {0167-6911},
   MRCLASS = {49J24 (49J05 49J15 70H30)},
  MRNUMBER = {2187840},
MRREVIEWER = {Matthias Kawski},
       URL = {https://doi.org/10.1016/j.sysconle.2005.06.002},
}

@article {PBG2022,
    AUTHOR = {Piro, Lorenzo and Mahault, Beno\^it and Golestanian, Ramin},
     TITLE = {Optimal navigation of microswimmers in complex and noisy
              environments},
   JOURNAL = {New J. Phys.},
  FJOURNAL = {New Journal of Physics},
    VOLUME = {24},
      YEAR = {2022},
     PAGES = {Paper No. 093037, 14},
      ISSN = {1367-2630},
   MRCLASS = {76Z05 (49N80 82C31)},
  MRNUMBER = {4515626},
       DOI = {10.1088/1367-2630/ac9079},
       URL = {https://doi.org/10.1088/1367-2630/ac9079},
}

@article{Serres2009,
 author = {Serres, Ulysse},
 title = {On {Zermelo}-like problems: {Gauss}-{Bonnet} inequality and {E}. {Hopf} theorem},
 fjournal = {Journal of Dynamical and Control Systems},
 journal = {J. Dyn. Control Syst.},
 issn = {1079-2724},
 volume = {15},
 number = {1},
 pages = {99--131},
 year = {2009},
 language = {English},
 doi = {10.1007/s10883-008-9056-6},
 zbMATH = {5840619},
 Zbl = {1203.34126}
}

@article {BCW2023,
    AUTHOR = {Bonnard, Bernard and Cots, Olivier and Wembe, Boris},
     TITLE = {Zermelo navigation problems on surfaces of revolution and
              geometric optimal control},
   JOURNAL = {ESAIM Control Optim. Calc. Var.},
  FJOURNAL = {ESAIM. Control, Optimisation and Calculus of Variations},
    VOLUME = {29},
      YEAR = {2023},
     PAGES = {Paper No. 60, 34},
      ISSN = {1292-8119,1262-3377},
   MRCLASS = {49K15 (53C60 70H05)},
  MRNUMBER = {4621417},
MRREVIEWER = {Mar\'ia\ Barbero-Li\~n\'an},
       DOI = {10.1051/cocv/2023052},
       URL = {https://doi.org/10.1051/cocv/2023052},
}

@book {Carath1967,
    AUTHOR = {Carath\'eodory, Constantin},
     TITLE = {Calculus of variations and partial differential equations of
              the first order. {P}art {II}: {C}alculus of variations},
      NOTE = {Translated from the German by Robert B. Dean, Julius J.
              Brandstatter, translating editor},
 PUBLISHER = {Holden-Day, Inc., San Francisco, Calif.-London-Amsterdam},
      YEAR = {1967},
     PAGES = {xvi+175-398},
   MRCLASS = {49.00 (35.00)},
  MRNUMBER = {232264},
}

@article{vM1931,
 author = {von Mises, Richard},
 title = {Zum {Navigationsproblem} der {Luftfahrt}},
 fjournal = {Zeitschrift f{\"u}r Angewandte Mathematik und Mechanik (ZAMM)},
 journal = {Z. Angew. Math. Mech.},
 issn = {0044-2267},
 volume = {11},
 pages = {373--381},
 year = {1931},
 language = {German},
 doi = {10.1002/zamm.19310110505},
 zbMATH = {3002700},
 Zbl = {0003.01201}
}

@article{LC1931,
 author = {Levi-Civita, Tullio},
 title = {{\"U}ber {Zermelos} {Luftfahrtproblem}},
 fjournal = {Zeitschrift f{\"u}r Angewandte Mathematik und Mechanik (ZAMM)},
 journal = {Z. Angew. Math. Mech.},
 issn = {0044-2267},
 volume = {11},
 pages = {314--322},
 year = {1931},
 language = {German},
 doi = {10.1002/zamm.19310110404},
 zbMATH = {3001683},
 Zbl = {0002.14403}
}

@book{attouch_variational_2006,
author = {Attouch, Hedy and Buttazzo, Giuseppe and Michaille, Gérard},
title = {Variational Analysis in Sobolev and BV Spaces},
publisher = {Society for Industrial and Applied Mathematics},
year = {2014},
doi = {10.1137/1.9781611973488},
address = {Philadelphia, PA},
edition   = {2nd},
}

@book{hiriart-urruty_fundamentals_2004,
	address = {Berlin Heidelberg},
	edition = {Corr. 2. print},
	series = {Grundlehren {Text} {Editions}},
	title = {Fundamentals of convex analysis},
	isbn = {978-3-540-42205-1},
	language = {ger eng},
	publisher = {Springer},
	author = {Hiriart-Urruty, Jean-Baptiste and Lemaréchal, Claude},
	year = {2004},
}

@book{bressan_introduction_2007,
	address = {Springfield},
	series = {American Institute of Mathematical Science. {Series} on Applied Mathematics},
	title = {Introduction to the Mathematical Theory of Control},
	isbn = {978-1-60133-002-4},
	language = {eng},
	publisher = {AIMS},
	author = {Bressan, Alberto and Piccoli, Benedetto},
	year = {2007},	
}

@book{aliprantis_infinite_2006,
	edition = {3},
	title = {Infinite {Dimensional} {Analysis}},
	isbn = {978-3-540-29586-0},
    language={eng},
	publisher = {Springer Berlin Heidelberg},
	author = {Aliprantis, Charalambos D.\  and Border, Kim C.},
	year = {2006},
    DOI={https://doi.org/10.1007/3-540-29587-9},
}

@book{perko_differential_2009,
	address = {New York, Berlin Heidelberg},
	edition = {3. ed.},
	series = {Texts in Applied Mathematics},
	title = {Differential Equations and Dynamical Systems},
	isbn = {978-0-387-95116-4},
	language = {en},
	number = {7},
	publisher = {Springer},
	author = {Perko, Lawrence Marion},
	year = {2009},
}















\end{document}